\documentclass{amsart}
\usepackage[all]{xy}

\newtheorem{theorem}{Theorem}
\newtheorem{lemma}{Lemma}
\newtheorem{proposition}{Proposition}
\newtheorem{corollary}{Corollary}

\newtheorem{definition}{Definition}

\theoremstyle{remark}
\newtheorem{remark}{Remark}[section]
\numberwithin{equation}{section}

\newcommand{\C}{{\mathbb C}}       
\newcommand{\R}{{\mathbb R}}       
\newcommand{\Z}{{\mathbb Z}}       
\newcommand{\N}{{\mathbb N}}

\newcommand{\internalcomment}[1]{}

\begin{document}

\title[Geometric conditions for an interpolation formula]
{Geometric conditions for the reconstruction of a holomorphic function by an interpolation formula}

\author{Amadeo Irigoyen}




\begin{abstract}

We give here some precisions and improvements about
the validity of the explicit reconstruction of any
holomorphic function on a ball of $\mathbb{C}^2$ from
its restrictions on a family of complex lines.
Such validity depends on the mutual distribution of
the lines. This condition can be geometrically
described and is equivalent to a stronger stability
of the reconstruction formula in terms of permutations
and subfamilies of these lines. The motivation of 
this problem comes from possible applications 
in mathematical economics and medical imaging.

\end{abstract}

\maketitle

\tableofcontents

\section{Introduction}

\subsection{Setting of the problem and some reminders}

In this paper we give some answers and improvements of the results
from~\cite{amadeo4}, where we deal with a special case of the
general problem of reconstruction of
a holomorphic function from its restrictions on a family of analytic
submanifolds. Here the setting is the following:
on the one hand, we consider for the analytic submanifolds any family of complex lines
in $\C^2$ that cross the origin.
Such a family can be written as
\begin{eqnarray}\label{deflines}
\left\{z\in\mathbb{C}^2,\,
z_1-\eta_jz_2=0
\right\}_{j\geq1}
\,,
\end{eqnarray}
where the directions $\eta_j\in\mathbb{C}$
are all different (we omit the special
line $\{z_2=0\}$).
On the other hand, let be $f\in\mathcal{O}\left(\C^2\right)$
(resp. $f\in\mathcal{O}\left(B_2\left(0,r_0\right)\right)$
where for any fixed $r_0>0$, 
$B_2\left(0,r_0\right)\subset\C^2$ is the complex ball defined as
\begin{eqnarray*}
B_2\left(0,r_0\right)
& = &
\left\{
z\in\C^2\,,\;
\left|z_1\right|^2+\left|z_2\right|^2<r_0^2
\right\}
\;)
\,.
\end{eqnarray*}
We then want to give an effective reconstruction of $f$ from
its restrictions on these complex lines. An application of some
methods from~\cite{berndtsson} yields the following interpolation formula,
that we remind from~\cite{amadeo3} and~\cite{amadeo4}:
\begin{eqnarray}\label{formula}
& &
E_N(f;\eta)(z)
\;:=\;
\sum_{p=1}^N
\left(
\prod_{j=p+1}^N
(z_1-\eta_j z_2)
\right)
\sum_{q=p}^N
\frac{1+\eta_p\overline{\eta_q}}{1+|\eta_q|^2}
\frac{1}
{\prod_{j=p,j\neq q}^N(\eta_q-\eta_j)}
\times
\\\nonumber
& &
\;\;\;\;\;\;\;\;\;\;\;\;\;\;\;\;\;\;\;\;\;\;\;\;
\times
\sum_{m\geq N-p}
\left(
\frac{z_2+\overline{\eta_q}z_1}{1+|\eta_q|^2}
\right)^{m-N+p}
\frac{1}{m!}
\frac{\partial^m}{\partial v^m}|_{v=0}
[f(\eta_qv,v)]
\,,
\end{eqnarray}
where $N\geq1$ and $z=(z_1,z_2)\in\mathbb{C}^2$.
We know (see Proposition~3 from~\cite{amadeo4}) that
for all $N\geq1$
and $f\in\mathcal{O}\left(\C^2\right)$
(resp. $f\in\mathcal{O}\left(B_2\left(0,r_0\right)\right)$), 
$E_N(f;\eta)$ is well-defined and satisfies the following properties:

\begin{itemize}

\item

$E_N(f;\eta)\in\mathcal{O}\left(\mathbb{C}^2\right)$
(resp. $E_N(f;\eta)\in\mathcal{O}\left(B_2(0,r_0)\right)$);

\item

$E_N(f;\eta)$ is an explicit formula that is constructed with the data
\begin{eqnarray*}
\left\{f_{|\{z_1=\eta_jz_2\}}\right\}_{1\leq j\leq N}
\;;
\end{eqnarray*}

\item

$\forall\,j=1,\ldots,N$,
$E_N(f;\eta)_{|\{z_1=\eta_jz_2\}}=f_{|\{z_1=\eta_jz_2\}}$;

\item

$\forall\,P\in\mathbb{C}[z_1,z_2]$ with
$\deg P\leq N-1$,
$E_N(P;\eta)\equiv P$.

\end{itemize}

One reason for the choice of a family of lines~(\ref{deflines}) is that it is
well suited for the methods in~\cite{berndtsson}, which readily produce
formula~(\ref{formula}). But the essential reason comes from
possible applications to the
real Radon transform theory, that may have
consequences in mathematical economics and medical imaging.
Indeed, let $\mu$ be a measure with compact support
$K\subset\R^2$ (w.l.o.g. one can assume that $0\in K$). 
We want to reconstruct it from
its Radon transforms on a finite number of directions, i.e. from
$(\mathcal{R}\mu)\left(\theta^{(j)},s\right)$ with
$\left(\theta^{(j)},s\right)\in\mathbb{S}^{1}\times\R$ and
$j=1,\ldots,N$, where $\mathbb{S}^{1}$ is the unit sphere of $\R^2$ and
\begin{eqnarray}\label{defradon}
(\mathcal{R}\mu)\left(\theta^{(j)},s\right) & := &
\frac{\partial}{\partial s}
\int_{\{x\in\R^2,\,\theta_1^{(j)}x_1+\theta_2^{(j)}x_2\leq\,s\}}
\mu(dx)
\;.
\end{eqnarray}
As it was explained at the Introduction of~\cite{amadeo4},
we consider the Fantappie transform $\Phi_{\mu}$ of $\mu$, that is defined
on the dual space
\\$K^{\star}:=
\left\{
\xi=[\xi_0:\xi_1:\xi_2]\in\C {\bf P}^2,\,
<\xi,x>\neq0,\,
\forall\,x\in K
\right\}$
and is holomorphic there.
Explicitly,
\begin{eqnarray*}
\Phi_{\mu}\;:\;\xi\in K^{\star}
& \mapsto &
<\mu,\dfrac{\xi_0}{<\xi,x>}>
\;:=\;\int_{x\in K}
\dfrac{\xi_0}{<\xi,x>}\mu(dx)
\end{eqnarray*}
(see~\cite{martineau}).
In addition, we know that there is $r_K>0$ such that for all
$\theta\in\mathbb{S}^{1}$ and all $u\in\C$ with $|u|<r_K$
(so that $[1:u\theta_1:u\theta_2]\in K^{\star}$),
\begin{eqnarray*}
\Phi_{\mu}([1:u\theta_1:u\theta_2]) & = &
\int_{-\infty}^{+\infty}
\frac{(\mathcal{R}\mu)(\theta,s)}{1+s\,u}\,ds
\,,
\end{eqnarray*}
i.e. the knowledge of 
$(\mathcal{R}\mu)(\theta^{(j)},s),\,j=1,\ldots,N,\,s\in\R$,
allows to know the restriction of
$\Phi_{\mu}\in\mathcal{O}\left(B_2(0,r_K)\right)$
on every line 
$L_{\theta^{(j)}}=\{(u\theta_1,u\theta_2),\,u\in\C\}
=\{z\in\C^2,\,z_1=\eta_jz_2\}$
where 
\begin{eqnarray}\label{etaradon}
\eta_j
& = &
\theta_1^{(j)}/\theta_2^{(j)}
\;\in\R,\;j=1,\ldots,N
\end{eqnarray}
(w.l.o.g. one can assume that $\theta^{(j)}_2\neq0$ for all $j=1,\ldots,N$).

The family of measures defined for $N\geq1$ by
$\mu_N:=\Phi^{-1}\left[E_N(\Phi_{\mu};\eta)\right]$
(where $E_N(\cdot;\eta)$ is the above fomula~(\ref{formula}),
and $\Phi^{-1}$ is the reciprocal isomorphism, whose
existence is guaranteed by~\cite{martineau}),
is interpolating in the meaning that
$<\mu_N,x_1^kx_2^l>=\\<\mu,x_1^kx_2^l>$
for all $N\geq1$ and $k,\,l\geq0$ with
$k+l\leq N$. 
Since by~(\ref{etaradon}) the set 
$\left\{\eta_j\right\}_{j\geq1}$ is a subset of $\R$,
by an application of Theorem~\ref{theorem2} below,
we will conclude that the family $\left\{\mu_N\right\}_{N\geq1}$
will approximate $\mu$ in an appropriate topology. 
In addition, an application of some results of Henkin and Shananin from~\cite{hensha2}
will allow to compute the reconstruction with good estimates.
These expected results are handled in~\cite{irigoyen5}, that is currently in progress.

\medskip

\subsection{Essential results}

The essential problem is that there is no guarantee that, as
$N\rightarrow\infty$, $E_N(f;\eta)$ will
converge to $f$ (although it coincides
with $f$ on an increasing number of lines).
We know from~\cite{amadeo4} that in general it is not the case,
i.e. there are families of lines with (at least)
an associated holomorphic function $f$ such that
$E_N(f;\eta)$ will not converge. 
Since we are interested in a reconstruction formula
whose convergence is guaranted for every holomorphic function
$f$,
we want to determine all the {\em good} families of lines 
$\eta=\left(\eta_j\right)_{j\geq1}$
for which the convergence of the associated interpolation
formula $E_N(\cdot;\eta)$ is guaranteed for every holomorphic function.
Theorems~1 and~2 from~\cite{amadeo4} give equivalent criteria
for the validity of such a reconstruction: roughly speaking, the sequence
of the directions $\left(\eta_j\right)_{j\geq1}$ of the lines~(\ref{deflines})
must satisfy an exponential estimate of their divided differences
(an operator of successive discrete derivatives, see for example~\cite{boor},
\cite{guelfond} and~\cite{steffensen1} for the definition
and essential results).
Nevertheless, the difficulty to check this condition on divided differences
gives us the motivation to find a criterion that is easier
to understand. This leads to the following definition:

\begin{definition}\label{intercurve}

The set $\{\eta_j\}_{j\geq1}$ is
{\em locally interpolable by real-analytic curves} if,
for all
$\zeta\in\overline{\{\eta_j\}_{j\geq1}}$
(the topological closure of
$\{\eta_j\}_{j\geq1}$ in $\C{\bf P}^1$),
there exist a neighborhood
$V$ of $\zeta$ and
a smooth real-analytic curve
$\mathcal{C}$ such that
$\zeta\in\mathcal{C}$ and
\begin{eqnarray}\label{geom}
V\bigcap\left\{\eta_j\right\}_{j\geq1}
& \subset &
\mathcal{C}
\,.
\end{eqnarray}

\end{definition}

This new geometric condition is a sufficient criterion for the convergence of the interpolation
formula $E_N(\cdot;\eta)$ and yields the following result, given as Theorem~3 from~\cite{amadeo4}.

\begin{theorem}\label{theorem2}

If $\{\eta_j\}_{j\geq1}$ is locally interpolable by real-analytic curves,
then the interpolation formula
$E_N(f,\eta)$ converges to $f$ uniformly on any compact
$K\subset\C^2$ and for all
$f\in\mathcal{O}\left(\C^2\right)$.

Similarly, $r_0$ being given, there is $\varepsilon_{\eta}>0$
such that, for all $f\in\mathcal{O}\left(B_2(0,r_0)\right)$,
$E_N(f;\eta)$ converges to $f$ uniformly on any compact subset
$K\subset B_2(0,\varepsilon_{\eta}r_0)$.

\end{theorem}

Nevertheless, this new criterion
is not equivalent. Indeed, as it has been suspected
in the Introduction of~\cite{amadeo4}, there are sequences of lines
that are not locally interpolable by real-analytic curves
and whose associated formula $E_N(\cdot;\eta)$ converges. 
Proposition~\ref{propctrex} below is the first result of this paper and
gives an explicit example of such a family:
it consists on constructing a sequence 
$(\eta_j)_{j\geq1}$ as the increasing union of $1/2^r$-nets, $r\geq0$, of the square
$[0,1]+i[0,1]=\left\{z\in\C,\,0\leq\Re(z),\,\Im(z)\leq1\right\}$
(so that for all $N\geq1$,
the first $N$ points $\eta_j$'s are the {\em most} separated
possible from each other).

\begin{proposition}\label{propctrex}

There exists (at least) one sequence
$\left(\eta_j\right)_{j\geq1}$ that is not locally interpolable by real-analytic curves
but whose associated interpolation
formula $E_N(\cdot;\eta)$ converges, i.e.
for all $f\in\mathcal{O}\left(\C^2\right)$, $E_N(\cdot;f)$ converges to $f$ uniformly on any
compact subset $K\subset \C^2$ (similarly, $r_0>0$ being fixed,
there is $\varepsilon_{\eta}>0$ such that for all
$f\in\mathcal{O}\left(B_2\left(0,r_0\right)\right)$,
$E_N(f;\eta)$ converges to $f$ uniformly on any compact subset
$K\subset B_2(0,\varepsilon_{\eta}r_0)$).

\end{proposition}

This first conclusion leads to the following question:
why is this geometric criterion not (always) necessary?
On the other hand, the expression~(\ref{formula}) of $E_N(\cdot;\eta)$
clearly involves the enumeration of the lines $\eta_j$'s.
Since Definition~\ref{intercurve} is a condition about sets that
does not depend on any of its enumerations, one is tempted into considerating
the action of the
group of permutations
$\mathfrak{S}_{\N}$ and check the validity of the convergence
of $E_N(\cdot;\sigma(\eta))$, where
the sequence $\sigma(\eta)$ is defined from $\eta=\left(\eta_j\right)_{j\geq1}$ by
\begin{eqnarray}
\sigma(\eta)
& = &
\left(\eta_{\sigma(j)}\right)_{j\geq1}
\end{eqnarray}
(in order to simplify the notation,
$\mathfrak{S}_{\N}$ will mean $\mathfrak{S}_{\N\setminus\{0\}}$
since all the considered sequences in the paper start by $j=1$).

Now $\sigma\in\mathfrak{S}_{\N}$ being given, 
one could first think that $E_N(f;\sigma(\eta))$ and $E_N(f;\eta)$ 
are essentially the {\em same}. Indeed, if $M_N:=\max\left\{N,\sigma(1),\ldots,\sigma(N)\right\}$,
then $E_{M_N}(f;\eta)$ and $E_{M_N}(f;\sigma(\eta))$
both interpolate $f$ on the $M_N$ first lines
$\eta_1,\ldots,\eta_N$ and 
$\eta_{\sigma(1)},\ldots,\eta_{\sigma(N)}$. Nevertheless,  if we change
the order of the sequence of the square from
Proposition~\ref{propctrex} above, the associated interpolation
formula may not converge anymore.
This leads to the following
question:
a given sequence $\left(\eta_j\right)_{j\geq1}$
whose associated interpolation formula 
$E_N(\cdot;\sigma(\eta))$ always
converges under the action of any permutation $\sigma$,
should be locally interpolable by real-analytical
curves? We will see that
the answer is affirmative as claimed by Theorem~\ref{fusion} below.

In order to deal with this problem, we need to consider the
following question: the sequence $\eta=\left(\eta_j\right)_{j\geq1}$
being fixed, 
if the formula $E_N(\cdot;\eta)$ converges,
what about
$E_N(\cdot;\eta')$, where 
$\eta':=\left(\eta_{j_k}\right)_{k\geq1}$ is any given
(infinite) subsequence of $\eta$?
On a first sight, the answer looks positive because of
the following intuitive argument: if $E_N(f;\eta)$ can interpolate
$f$ on {\em more} lines than $E_{N'}(f;\eta')$ does
(where $N'$ is the number of $k\geq1$ such that $j_k\leq N$) and
$E_N(f;\eta)$ converges, then why should not
$E_N(f;\eta')$ too? The true answer is that this heuristic argument is false.
Indeed, 
 it is also a strong condition that is equivalent
to the geometric criterion~(\ref{geom}).
This claim and the above one are specified by the following result, that
is the main theorem of this paper.

\begin{theorem}\label{fusion}

Let $\eta=\left(\eta_j\right)_{j\geq1}$ be any sequence
defined as in~(\ref{deflines}). The following conditions are equivalent:

\begin{enumerate}

\item\label{fusiongeom}

the set 
$\left\{\eta_j\right\}_{j\geq1}$ is locally interpolable by real-analytic curves;

\item\label{fusionpermut}

for all $f\in\mathcal{O}\left(\C^2\right)$
and all $\sigma\in\mathfrak{S}_{\N}$, 
$E_N(f;\sigma(\eta))$ converges to $f$
uniformly on any compact subset $K\subset\C^2$;

\item\label{fusionsubseq}

for all $f\in\mathcal{O}\left(\C^2\right)$ and all subsequence
$\eta'=\left(\eta_{j_k}\right)_{k\geq1}$,
$E_N\left(f;\eta'\right)$
converges to $f$ uniformly on any compact subset $K\subset\C^2$.

\end{enumerate}

\end{theorem}

First, this result finally gives an equivalence, for a given sequence $\left(\eta_j\right)_{j\geq1}$,
between the strong geometric
hypothesis~(\ref{geom}) and sharper conditions 
in terms of the validity of the
convergence of the associated interpolation formula $E_N(\cdot;\eta)$.
In particular, it clarifies in which sense this geometric condition
is sufficient.

Next, this results only deals with the convergence of $E_N(f;\eta)$ for
any $f\in\mathcal{O}\left(\C^2\right)$ and we would like to know what happens
if we consider the same assertion with any $f\in\mathcal{O}\left(B_2\left(0,r_0\right)\right)$
for any fixed $r_0>0$. One of the applications of Theorem~\ref{fusion} is its generalization
to the case of every complex ball $B_2(0,r_0)$, as specified by the following result.

\begin{corollary}\label{corollaireducorollaire}

Let $\eta=\left(\eta_j\right)_{j\geq1}$ be any sequence defined as in~(\ref{deflines}) and let
be $r_0>0$. The following conditions are equivalent:

\begin{enumerate}

\item\label{corolgeom}

the set 
$\left\{\eta_j\right\}_{j\geq1}$ is locally interpolable by real-analytic curves;

\item\label{corolpermut}

there is $\varepsilon_{\eta}>0$ such that 
for all $f\in\mathcal{O}\left(B_2(0,r_0)\right)$
and all $\sigma\in\mathfrak{S}_{\N}$, 
$E_N(f;\sigma(\eta))$ converges to $f$
uniformly on any compact subset
$K\subset B_2\left(0,\varepsilon_{\eta}r_0\right)$;

\item\label{corolsubseq}

there is $\varepsilon_{\eta}>0$ such that 
for all $f\in\mathcal{O}\left(B_2(0,r_0)\right)$
and all subsequence
$\eta'=\left(\eta_{j_k}\right)_{k\geq1}$,
$E_N\left(f;\eta'\right)$
converges to $f$ uniformly on any compact subset 
$K\subset B_2\left(0,\varepsilon_{\eta}r_0\right)$.

\end{enumerate}

\end{corollary}

As it can be noticed, these results are equivalences between a geometric condition
and the validity
of the convergence of its associated interpolation formula $E_N(\cdot;\eta)$
(i.e. in terms of functional approximation theory).

Moreover, we have another consequence that gives some precision
on the speed of convergence of $E_N(f;\eta)$ to $f$.

\begin{corollary}\label{velocitat}

When any of the equivalent conditions from Theorem~\ref{fusion}
or Corollary~\ref{corollaireducorollaire} is fulfilled, one has in addition the following estimate:
for all $\mathcal{K}\subset\mathcal{O}\left(\C^2\right)$
(resp. $\mathcal{K}\subset\mathcal{O}\left(B_2(0,r_0)\right)$)
and $K\subset\C^2$ (resp. $K\subset B_2(0,\varepsilon_{\eta}r_0)$) 
compact subsets, there are
$C_{\mathcal{K},K}$ and $\varepsilon_K>0$ such that for all
$\sigma\in\mathfrak{S}_{\N}$,
for all $\eta'=\left(\eta_{j_k}\right)_{k\geq1}$
and all $N\geq1$,
\begin{eqnarray*}
\sup_{f\in\mathcal{K}}
\;
\sup_{z\in K}
\left|
f(z)-E_N(f;\sigma(\eta))(z)
\right|
& \leq &
C_{\mathcal{K},K}\,(1-\varepsilon_K)^N
\end{eqnarray*}
and
\begin{eqnarray*}
\sup_{f\in\mathcal{K}}
\;
\sup_{z\in K}
\left|
f(z)-E_N\left(f;\eta'\right)(z)
\right|
& \leq &
C_{\mathcal{K},K}\,(1-\varepsilon_K)^N
\,.
\end{eqnarray*}

\end{corollary}

In particular, as it has been pointed out in~\cite{amadeo4}, a simple convergence
of $E_N(\cdot;\eta)$
(i.e. convergence of $E_N(f;\eta)$
for every fixed holomorphic function $f$) implies a uniform one. This can be interpreted as a
Banach-Steinhaus property for the family of operators
$\left\{E_N\left(\cdot;\eta\right)\right\}_{N\geq1}$ in the canonical
topology for the holomorphic functions (i.e. the topology of uniform convergence
on any compact subset).

Finally, the essential argument for the proof of Theorem~\ref{fusion} 
follows from the following result, whose proof is given in Section~\ref{subsequaux}.

\begin{proposition}\label{norealint}

Let $\left(\eta_j\right)_{j\geq1}$ be any sequence such that
the set $\left\{\eta_j\right\}_{j\geq1}$
is not locally interpolable by real-analytic curves.
Then there exists a subsequence
$\left(\eta_{j_k}\right)_{k\geq1}$ of
$\left(\eta_j\right)_{j\geq1}$ that satisfies
the following conditions:

\begin{itemize}

\item

the sequence $\left(\eta_{j_k}\right)_{k\geq1}$
is convergent in $\C{\bf P}^1$;

\item

the set $\left\{\eta_{j_k}\right\}_{k\geq1}$
is not locally interpolable by real-analytic curves.

\end{itemize}

\end{proposition}

We know that if $\left\{\eta_j\right\}_{j\geq1}$ (coming from any
sequence $\left(\eta_j\right)_{j\geq1}$)
is locally interpolable by real-analytic curves,
then so will be any of its subsets (finite or infinite),
in particular if it comes from a convergent subsequence 
$\left(\eta_{j_k}\right)_{k\geq1}$.
Conversely, we may ask what happens if $\left\{\eta_j\right\}_{j\geq1}$ is not
locally interpolable by real-analytic curves.
Proposition~\ref{norealint} gives an affirmative answer, in the sense that
one can extract convergent subsequences that are still not
locally interpolable by real-analytic curves.

\bigskip

I would like to thank G. Henkin for having introduced me this
interesting problem and J. Ortega-Cerd\`a for all his
guidance, rewarding ideas and discussions to progress in.

\bigskip

\section{On the non-equivalence of the geometric criterion}\label{ctrex}

In this section we deal with the proof of Proposition~\ref{propctrex}. We first need some
reminders and preliminary results.

\subsection{Some reminders}

First, the following result is given as Lemma~20
from~\cite{amadeo4} and is a necessary condition for a set
to satisfy the geometric condition~(\ref{geom}).

\begin{lemma}\label{interdense}

The topological closure of a set that is
locally interpolable by real-analytic curves, has empty interior.

\end{lemma}

Next, we remind Theorem~1 as one of the essential results from~\cite{amadeo4}
and that gives an 
equivalent criterion for a bounded sequence $\left(\eta_j\right)_{j\geq1}$
to make converge its associated interpolation formula $E_N(\cdot;\eta)$.

\begin{theorem}\label{theorem1}

Let $\left(\eta_j\right)_{j\geq1}$ be bounded and fix any $r_0>0$.
The following conditions are equivalent:

\begin{enumerate}

\item\label{thm11}

there is $\varepsilon_{\eta}>0$ such that,
for all $f\in\mathcal{O}\left(B_2(0,r_0)\right)$, the interpolation formula
$E_N(f;\eta)$ converges to $f$, uniformly on any
compact subset of $B_2(0,\varepsilon_{\eta}r_0)$;

\item\label{thm12}

for all $g\in\mathcal{O}\left(\C^2\right)$,
the interpolation formula $E_N(g;\eta)$ converges to $g$,
uniformly on any compact subset of $\C^2$;

\item\label{thm13}

$\exists\,R_{\eta}\geq1$, 
$\forall\,p,q\geq0$,
\begin{eqnarray}\label{criter}
\left|
\Delta_{p,(\eta_p,\ldots,\eta_1)}
\left[
\left(
\frac{\overline{\zeta}}{1+|\zeta|^2}
\right)^q
\right]
\left(\eta_{p+1}\right)
\right|
& \leq &
R_{\eta}^{p+q}
\,.
\end{eqnarray}

\end{enumerate}

\end{theorem}

The operator $\Delta_p$ is called {\em divided differences} and is defined as follows
(for any application $h$ that is defined at the $\eta_j$'s):
\begin{eqnarray}\label{divdiff}
& &
\begin{cases}
\Delta_0(h)(\eta_1)=h(\eta_1)\,;
\\
\\
\mbox{for all }\,p\geq1\,,\;
\Delta_{p,(\eta_p,\ldots,\eta_1)}
(h)(\eta_{p+1})
\;=
\\
\;\;\;\;\;\;\;\;\;\;\;\;\;\;\;\;\;\;\;\;\;\;\;\;
=\;
\dfrac{\Delta_{p-1,(\eta_{p-1},\ldots,\eta_1)}(h)(\eta_{p+1})
-\Delta_{p-1,(\eta_{p-1},\ldots,\eta_1)}(h)(\eta_{p})}
{\eta_{p+1}-\eta_p}
\,.
\end{cases}
\end{eqnarray}
$\Delta_p(h)$ can be seen as the discrete derivative of order
$p$ of the function $h$.
A lot its properties
can be found in the references, in particular the following one (see for example~\cite{devore},
Chapter~4, 7~(7.7)).

\begin{lemma}\label{Deltapaux}

Let $\{\eta_j\}_{j\geq1}$ be any set of different
points and $h$ any function defined on them. 
One has for all $p\geq0$,
\begin{eqnarray*}
\Delta_{p,(\eta_p,\ldots,\eta_1)}
[h]
\left(\eta_{p+1}\right)
& = &
\sum_{q=1}^{p+1}
\frac{h\left(\eta_q\right)}
{\prod_{j=1,j\neq q}^{p+1}
\left(\eta_q-\eta_j\right)}
\,.
\end{eqnarray*}

\end{lemma}

We also deduce as an application the following result that will be useful for the proof of Theorem~\ref{fusion}.

\begin{corollary}\label{corollairedulemme}

For all $p,\,q\geq0$, 
\begin{eqnarray*}
\left|
\Delta_{p,(\eta_p,\ldots,\eta_1)}
\left[\zeta\mapsto\left(\frac{\overline{\zeta}}{1+|\zeta|^2}\right)^q\right]
\left(\eta_{p+1}\right)
\right|
& \leq &
\sum_{l=1}^{p+1}
\frac{1}
{\prod_{j=1,j\neq l}^{p+1}
\left(\eta_l-\eta_j\right)}
\,.
\end{eqnarray*}
In particular, the bound does not depend on $q\geq0$.

\end{corollary}

\begin{proof}

The proof immediately follows by Lemma~\ref{Deltapaux} with the particular choice of
$h(\zeta)=\left(\dfrac{\overline{\zeta}}{1+|\zeta|^2}\right)^q$
since for all $\zeta\in\C$, one has that
\begin{eqnarray*}
\left|
\left(\frac{\overline{\zeta}}{1+|\zeta|^2}\right)^q
\right|
& = &
\left(\frac{|\zeta|}{1+|\zeta|^2}\right)^q
\;=\;
\left(\sqrt{\frac{|\zeta|^2}{1+|\zeta|^2}}\right)^q
\times
\frac{1}{\left(\sqrt{1+|\zeta|^2}\right)^q}
\;\leq\;
1
\,.
\end{eqnarray*}

\end{proof}

\medskip

\subsection{Construction of a counterexample}

In this subsection, we construct the explicit sequence
$\left(\eta_j\right)_{j\geq1}\subset\mathcal{Q}$, where
$\mathcal{Q}$ is the closed square
\begin{eqnarray}
\mathcal{Q}
& = &
[0,1]+i[0,1]
\;=\;
\left\{z\in\C\,,\;
0\leq\Re(z),\,\Im(z)\leq1
\right\}
\,.
\end{eqnarray}
Its essential required property is that the $\eta_j$'s must be the {\em most} 
separated possible from each other.
We start by setting
$\eta_1=0,\,\eta_2=1,\,\eta_3=1+i,\,\eta_4=i$. We find the maximal
number of points of $\mathcal{Q}$ whose mutual distance is not
smaller than $1$. When it is not possible anymore,
we add the maximal number of points whose mutual distance
is at least $1/2$, then
$\eta_5=1/2,\,\eta_6=i/2,\,\eta_7=(1+i)/2,\,\eta_8=1+i/2,
\eta_9=1/2+i$.
More generally, we will choose by induction on $r\geq0$
the maximal number of points whose mutual distance
is at least $1/2^r$. 

Let fix $r\geq0$ and let 
$\mathcal{A}_r$ be an $1/2^r$-net of $\mathcal{Q}$, i.e.
a set of points
that are at least at a distance of $1/2^r$ from each other.
One can choose
\begin{eqnarray}\label{defAr}
\mathcal{A}_r 
& = &
\left\{
\frac{s+it}{2^r}\,,\,
(s,t)\in\N^2
\,,\,
0\leq s,t\leq2^r
\right\}
\,,
\end{eqnarray}
whose cardinal is $\left(1+2^r\right)^2$
(one can check that
$\mathcal{A}_0=\left\{0,1,i,1+i\right\}=\left\{\eta_1,\eta_2,\eta_3,\eta_4\right\}$).
In addition, one has
the sequence of inclusions
\begin{eqnarray}
\mathcal{A}_0
\;\subset\;
\mathcal{A}_1
\;\subset\;
\cdots
\;\subset\;
\mathcal{A}_r
\;\subset\;
\cdots
\,.
\end{eqnarray}
The sequence
$\eta=\left(\eta_j\right)_{j\geq1}$ will be defined
by induction on $r\geq0$ as follows: we first choose 
$\eta_1,\eta_2,\eta_3$ and $\eta_4$ for the first set
$\mathcal{A}_0$ (notice that we do not specify any enumeration
for these first $\eta_j$'s); next, if we assume having constructed
$\eta_1,\ldots,\eta_{N_r}$ with 
\begin{eqnarray}\label{defNr}
N_r
& = &
\left(1+2^r\right)^2
\,,
\end{eqnarray}
we define $\eta_{N_r+1},\ldots,\eta_{N_{r+1}}$ so that
\begin{eqnarray}\label{numetajctrex}
\left\{\eta_{N_r+1},\ldots,\eta_{N_{r+1}}\right\}
& = &
\mathcal{A}_{r+1}\setminus\mathcal{A}_r
\,.
\end{eqnarray}
Again, the enumeration for these $\eta_j$'s does not matter. The only important fact
is that $\eta_j\in\mathcal{A}_{r_j}$ for all $j\geq1$, where $r_j$ is the
first $r\geq0$ such that $j\leq N_{r}$.

The sequence $\left(\eta_j\right)_{j\geq1}$
can be defined by induction on $r\geq0$ and one has
\begin{eqnarray*}
\left\{\eta_j\right\}_{j\geq1}
& = &
\mathcal{A}_{\infty}
\;:=\;
\bigcup_{r\geq0}
\mathcal{A}_r
\,.
\end{eqnarray*}
As it has been specified, the enumeration of $\left(\eta_j\right)_{j\geq1}$
does not matter as long as one has the following important condition:
for all $r\geq0$ and all $j,\,k\geq1$ such that $\eta_j\in\mathcal{A}_r$
and $\eta_k\in\mathcal{A}_{r+1}\setminus\mathcal{A}_r$,
then one necessarily has $j<k$.
Equivalently,
for all $r\geq0$,
the first $N_r$ points $\eta_j$'s
belong to $\mathcal{A}_r$.

We can deduce the following preliminar result.

\begin{lemma}\label{(etaj)dense}

The sequence $\left(\eta_j\right)_{j\geq1}$ is well-defined and is dense in $\mathcal{Q}$.
In addition, $\left(\eta_j\right)_{j\geq1}$ satisfies the following condition:
\begin{eqnarray}
\forall\;r\geq0\,,\;
\forall\;j\leq N_r=\left(1+2^r\right)^2,\;
\eta_j\in\mathcal{A}_r
\,.
\end{eqnarray}

\end{lemma}

\begin{proof}

The last assertion immediately follows from~(\ref{defNr}) and~(\ref{numetajctrex}).
In order to prove the density, let consider $z\in\mathcal{Q}$, $\varepsilon>0$ and
let be $r\geq0$ such that
$1/2^{r}\leq\varepsilon$. 
There is $\eta_{j_z}\in\mathcal{A}_{r+1}$ such that
$\left|\Re\left(\eta_{j_z}\right)-\Re(z)\right|\leq1/2^{r+1}$
and
$\left|\Im\left(\eta_{j_z}\right)-\Im(z)\right|\leq1/2^{r+1}$,
then
$\left|\eta_{j_z}-z\right|\leq\sqrt{2}/2^{r+1}<1/2^r\leq\varepsilon$.

\end{proof}

In order to prove Proposition~\ref{propctrex},
we first need to give an estimate of the divided differences 
$\left\{\Delta_{p}\right\}_{p\geq1}$ associated
with $\left(\eta_j\right)_{j\geq1}$.

\medskip

\subsection{A bound for the associated divided differences}

We start by the following preliminar result that is a lower bound for the products
that appear on the expression of the $\Delta_p$'s
given by Lemma~\ref{Deltapaux}.

\begin{lemma}\label{estimproduct}

Let consider the sequence $\left(\eta_j\right)_{j\geq1}$ from Lemma~\ref{(etaj)dense}.
There is $P_{\eta}\geq2$ 
such that, for all
$p\geq P_{\eta}$ and all
$q=1,\ldots,p+1$, one has
\begin{eqnarray*}
\prod_{j=1,j\neq q}^{p+1}
\left|
\eta_q-\eta_j
\right|
& \geq &
\exp(-9p)
\,.
\end{eqnarray*}

\end{lemma}

\begin{proof}

Let fix any $p\geq2$ and let consider the unique
$r\geq0$ such that
\begin{eqnarray}\label{encadrep}
\left(1+2^{r-1}\right)^2
\;<\;p+1\;\leq\;
\left(1+2^r\right)^2
\,.
\end{eqnarray}
Now let fix $\eta_q$ with $q=1,\ldots,p+1$, i.e.
\begin{eqnarray*}
\eta_q
\;=\;
\dfrac{s_q+it_q}{2^r}
& \mbox{ where } &
0\;\leq\; s_q,\,t_q\;\leq\;2^r
\,.
\end{eqnarray*}
Similarly,
for all $\eta_j\neq\eta_q$ with $j=1,\ldots,p+1$,
one has $\eta_j=\dfrac{s_j+it_j}{2^r}$ with
$0\leq s_j,\,t_j\leq2^r$
and $\left(s_j,t_j\right)\neq\left(s_q,t_q\right)$.
Since $\left|s_j-s_q\right|\leq2^r$
(resp. $\left|t_j-t_q\right|\leq2^r$),
then
\begin{eqnarray}\label{defkqj}
k_{q,j}
\;:=\;
\max\left\{
\left|s_j-s_q\right|
,
\left|t_j-t_q\right|
\right\}
& \in &
\left\{1,\ldots,2^r\right\}
\,.
\end{eqnarray}
It follows that
$\eta_j\in\mathcal{D}_{k_{q,j}}\left(\eta_q\right)$, where
$\mathcal{D}_k\left(\eta_q\right)$ is defined for all
$k\in\N$ by
\begin{eqnarray*}
\mathcal{D}_k\left(\eta_q\right)
& := &
\left\{
z\,=\,\frac{s+it}{2^r}\,,\;
(s,t)\in\Z^2\,,\;
\max\left\{
\left|s-s_q\right|
,
\left|t-t_q\right|
\right\}
\,=\,k
\right\}
\,.
\end{eqnarray*}

We first want to estimate 
$\mbox{card}\left[\mathcal{D}_k\left(\eta_q\right)\cap\left\{\eta_j,\,1\leq j\leq p+1,\,j\neq q\right\}\right]$
for all $k=1,\ldots,2^r$. We start by noticing that
\begin{eqnarray}\label{DkxSk}
\mathcal{D}_k\left(\eta_q\right)
& = &
\mathcal{S}_k\left(\eta_q\right)\setminus\mathcal{S}_{k-1}\left(\eta_q\right)
\,,
\end{eqnarray}
where $\mathcal{S}_k\left(\eta_q\right)$ is defined for all $k\in\N$ by
\begin{eqnarray*}
\mathcal{S}_k\left(\eta_q\right)
& := &
\left\{
z\,=\,\frac{s+it}{2^r}\,,\;
(s,t)\in\Z^2\,,\;
\max\left\{
\left|s-s_q\right|
,
\left|t-t_q\right|
\right\}
\,\leq\,k
\right\}
\,.
\end{eqnarray*}
Since on the one hand, one has for all $k\geq0$, that
\begin{eqnarray*}
\mbox{card}\left[
\mathcal{S}_k\left(\eta_q\right)
\right]
& = &
\mbox{card}\left[
\mathcal{S}_k\left(0\right)\right]
\;=\;
\mbox{card}
\left\{
z\,=\,\frac{s+it}{2^r}\,,\;
(s,t)\in\Z^2\,,\;
-k\leq s,\,t\leq k
\right\}
\\
& = &
(2k+1)^2
\,,
\end{eqnarray*}
and on the other hand, 
$\mathcal{S}_{k-1}\left(\eta_q\right)\subset\mathcal{S}_k\left(\eta_q\right)$ for all $k\geq1$,
it follows by~(\ref{DkxSk}) that
\begin{eqnarray}\nonumber
\mbox{card}\left[\mathcal{D}_k\left(\eta_q\right)\cap\left\{\eta_j,\,1\leq j\leq p+1,\,j\neq q\right\}\right]
& \leq &
\mbox{card}\left[\mathcal{D}_k\left(\eta_q\right)\right]
\\\nonumber
& = &
\mbox{card}\left[\mathcal{S}_k\left(\eta_q\right)\right]
-
\mbox{card}\left[\mathcal{S}_{k-1}\left(\eta_q\right)\right]
\\\label{estimcardDk}
& = &
(2k+1)^2-(2k-1)^2
\;=\;
8k
\,.
\end{eqnarray}

Next, one has for all $k=1,\ldots,2^r$, and all
$\eta_j\in\mathcal{D}_k\left(\eta_q\right)$ (with
$\eta_j=\dfrac{s_j+it_j}{2^r}$~), 
\begin{eqnarray}\label{estimetaj-etaq}
\left|
\eta_j-\eta_q
\right|
& \geq &
\max\left\{
\frac{\left|s_j-s_q\right|}{2^r}
,
\frac{\left|t_j-t_q\right|}{2^r}
\right\}
\;=\;
\frac{k}{2^r}
\,.
\end{eqnarray}

Finally, the estimates~(\ref{estimetaj-etaq}) and~(\ref{estimcardDk}) together yield
for all $k=1,\ldots,2^r$,
\begin{eqnarray*}
\prod_{\mathcal{D}_k\left(\eta_q\right)\cap\left\{\eta_j,\,1\leq j\leq p+1,\,j\neq q\right\}}
\left|\eta_q-\eta_j\right|
\;\geq\;
\left(
\frac{k}{2^r}
\right)^{\mbox{card}\left[\mathcal{D}_k\left(\eta_q\right)\cap\left\{\eta_j,\,1\leq j\leq p+1,\,j\neq q\right\}\right]}
\;\geq\;
\left(\frac{k}{2^r}\right)^{8k}
\,,
\end{eqnarray*}
the second inequality being valid since $0<k/2^r\leq1$.
By applying the following partition (justified by~(\ref{encadrep}) and
Lemma~\ref{(etaj)dense}),
\begin{eqnarray*}
\left\{
\eta_j\,,\;
j=1,\ldots,p+1\,,\;
j\neq q
\right\}
& = &
\bigcup_{k=1}^{2^r}
\left[
\mathcal{D}_k\left(\eta_q\right)\cap\left\{\eta_j,\,1\leq j\leq p+1,\,j\neq q\right\}
\right]
\end{eqnarray*}
(one indeed has $1\leq k\leq2^r$ by~(\ref{defkqj})),
we can deduce that
\begin{eqnarray}\nonumber
\prod_{j=1,j\neq q}^{p+1}
\left|
\eta_q-\eta_j
\right|
& = &
\prod_{k=1}^{2^r}
\left[
\prod_{\mathcal{D}_k\left(\eta_q\right)\cap\left\{\eta_j,\,1\leq j\leq p+1,\,j\neq q\right\}}
\left|\eta_q-\eta_j\right|
\right]
\;\geq\;
\prod_{k=1}^{2^r}
\left(\frac{k}{2^r}\right)^{8k}
\\\label{estimprodetaq}
& = &
\exp
\left[
\sum_{k=1}^{2^r}
8k\ln\left(k/2^r\right)
\right]
\;=\;
\exp
\left[
2^{2r+3}
\times
\frac{1}{2^r}
\sum_{k=1}^{2^r}
\frac{k}{2^r}
\ln\left(\frac{k}{2^r}\right)
\right]
\,.
\end{eqnarray}
The last expression involves the Riemann's sum of the continuous
function\\
$t\in\,]0,1]\mapsto t\ln t$, $0\mapsto0$, whose integral is
$\int_0^1t\ln tdt 
=\left[\dfrac{t^2}{2}\ln t\right]_0^1
-\int_0^1\dfrac{t}{2}dt
=-\dfrac{1}{4}$.
Then
\begin{eqnarray}\label{riemannsum}
& &
\frac{2^{2r+3}}{2^r}
\sum_{k=1}^{2^r}
\frac{k}{2^r}
\ln\left(\frac{k}{2^r}\right)
\;=\;
2^{2r+3}\left(-1/4+\varepsilon(1/r)\right)
\;=\;
\frac{2^{2r}}{4}
\left(-8+\varepsilon(1/r)\right)
\,,
\end{eqnarray}
where $\varepsilon(1/r)\rightarrow0$ as $1/r\rightarrow0$.
On the other hand, one has by~(\ref{encadrep}) that
\begin{eqnarray}\label{minorpxr}
p 
& \geq &
\left(1+2^{r-1}\right)^2-1
\;=\;
2^{2r-2}+2^r
\;\geq\;
\frac{2^{2r}}{4}
\,.
\end{eqnarray}
In addition, (\ref{encadrep}) also gives that 
$\varepsilon(1/r)=\varepsilon(1/p)$.
It follows by applying~(\ref{estimprodetaq}), (\ref{riemannsum}) and~(\ref{minorpxr})
that
\begin{eqnarray*}
\prod_{j=1,j\neq q}^{p+1}
\left|
\eta_q-\eta_j
\right|
& \geq &
\exp
\left[
p\times\left(-8+\varepsilon(1/p)\right)
\right]
\;\geq\;
\exp\left(-9p\right)
\,,
\end{eqnarray*}
for all $p\geq P_{\eta}$ ($\geq2$) so that
$\left|\varepsilon(1/p)\right|\leq1$
(notice that $P_{\eta}$ does not depend on $q=1,\ldots,p+1$).
The estimate being true for all $q=1,\ldots,p+1$,
the proof of the lemma is achieved.

\end{proof}

This allows us to prove the following result that will be useful for
the proof of Proposition~\ref{propctrex}.

\begin{lemma}\label{deltagen}

Let $h$ be any function defined on the set
$\{\eta_j\}_{j\geq1}$ (coming from the sequence
$\left(\eta_j\right)_{j\geq1}$ of Lemma~\ref{(etaj)dense}) and that is bounded:
\begin{eqnarray*}
\|h\|_{\infty}\;:=\;
\sup_{j\geq1}
\left|h\left(\eta_j\right)\right|
& < &
+\infty\,.
\end{eqnarray*}
Then there is $R_{\eta}\geq1$ 
such that for all $p\geq0$,
\begin{eqnarray*}
\left|
\Delta_{p,(\eta_p,\ldots,\eta_1)}
[h]\left(\eta_{p+1}\right)
\right|
& \leq &
\|h\|_{\infty}R_{\eta}^p
\,.
\end{eqnarray*}

\end{lemma}

\medskip

\begin{proof}

Let be $p\geq0$. One has by
Lemma~\ref{Deltapaux} that
\begin{eqnarray*}
\left|
\Delta_{p,(\eta_p,\ldots,\eta_1)}
[h]
\left(\eta_{p+1}\right)
\right|
\;\leq\;
\sum_{q=1}^{p+1}
\frac{\left|h\left(\eta_q\right)\right|}
{\prod_{j=1,j\neq q}^{p+1}
\left|\eta_q-\eta_j\right|}
\;\leq\;
\frac{
(p+1)
\|h\|_{\infty}}
{\min_{1\leq q\leq p+1}
\left(
\prod_{j=1,j\neq q}^{p+1}
|\eta_q-\eta_j|
\right)
}
\,.
\end{eqnarray*}

If $p\leq P_{\eta}-1$, then
\begin{eqnarray*}
\left|
\Delta_{p,(\eta_p,\ldots,\eta_1)}
[h]
\left(\eta_{p+1}\right)
\right|
& \leq &
C_{\eta}
\|h\|_{\infty}
\,,
\end{eqnarray*}
where
\begin{eqnarray*}
C_{\eta} & := &
\frac{P_{\eta}}
{\min_{1\leq p\leq P_{\eta}-1}
\left(
\min_{1\leq q\leq p+1}
\prod_{j=1,j\neq q}^{p+1}
|\eta_q-\eta_j|
\right)}
\,.
\end{eqnarray*}

Otherwise $p\geq P_{\eta}$ ($\geq2$) then one has 
by Lemma~\ref{estimproduct} that
\begin{eqnarray*}
\min_{1\leq q\leq p+1}
\left(
\prod_{j=1,j\neq q}^{p+1}
|\eta_q-\eta_j|
\right)
& \geq & 
1/\exp(9p)
\,,
\end{eqnarray*}
thus
\begin{eqnarray*}
\left|
\Delta_{p,(\eta_p,\ldots,\eta_1)}
[h]
\left(\eta_{p+1}\right)
\right|
\,\leq\,
(p+1)\|h\|_{\infty}
\exp(9p)
\,\leq\,
\exp(p)\|h\|_{\infty}
\exp(9p)
\,=\,
\|h\|_{\infty}
\exp(10p)
\end{eqnarray*}
(the second estimate being justified by the classical one:
$1+t\leq\exp(t)\,,\;\forall\,t\in\R$).

It follows that for all
$p\geq1$, one has
\begin{eqnarray*}
\left|
\Delta_{p,(\eta_p,\ldots,\eta_1)}
[h]
\left(\eta_{p+1}\right)
\right|
& \leq &
\|h\|_{\infty}\times
\max\left[C_{\eta}\,,\,\exp\left(10p\right)\right]
\\
& \leq &
\|h\|_{\infty}
\left(1+C_{\eta}\right)
\exp(10p)
\;\leq\;
\|h\|_{\infty}
R_{\eta}^p
\,,
\end{eqnarray*}
where
$R_{\eta}:=(1+C_{\eta})e^{10}$
(for $p=0$, one just has that
$\left|\Delta_0(h)\left(\eta_1\right)\right|
=\left|h\left(\eta_1\right)\right|\leq\|h\|_{\infty}\times R_{\eta}^0$)
and the proof is achieved.

\end{proof}

\begin{remark}

Notice that we do not need to assume any kind of regularity
for the function $h$, else that it is bounded on the set
$\{\eta_j\}_{j\geq1}$. 

\end{remark}

\medskip

\subsection{Proof of Proposition~\ref{propctrex}}

Now we can give the proof of the proposition.

\begin{proof}

First, the set
$\left\{\eta_j\right\}_{j\geq1}$ is dense in the square $\mathcal{Q}$
by Lemma~\ref{(etaj)dense} then its topological closure
$\overline{\left\{\eta_j\right\}_{j\geq1}}$ has nonempty
interior. It follows by Lemma~\ref{interdense} that
$\left\{\eta_j\right\}_{j\geq1}$ cannot be locally interpolable by real-analytic curves.

Next, the sequence $\left(\eta_j\right)_{j\geq1}$ being
bounded, in order to prove
that $E_N(\cdot;\eta)$ converges
(i.e. for entire functions as well as for holomorphic functions
on any fixed ball $B_2\left(0,r_0\right)$), it suffices
to show that $\left(\eta_j\right)_{j\geq1}$ satisfies 
the estimate~(\ref{criter}) in Theorem~\ref{theorem1}.
For all $q\geq0$, one has with the choice of
$h(\zeta)=\left(\dfrac{\overline{\zeta}}{1+|\zeta|^2}\right)^q$ that
\begin{eqnarray*}
\left\|
\left(
\frac{\overline{\zeta}}{1+|\zeta|^2}
\right)^q
\right\|_{\infty}
& \leq &
\left\|
\left(\sqrt{\frac{\left|\overline{\zeta}\right|^2}{1+|\zeta|^2}}\right)^q
\right\|_{\infty}
\;\times\;
\left\|
\frac{1}{\left(\sqrt{1+|\zeta|^2}\right)^q}
\right\|_{\infty}
\\
& \leq &
\left\|
\frac{\sqrt{1+|\zeta|^2}}{1+|\zeta|^2}
\right\|_{\infty}^q
\;\times\;
1
\;\leq\;
1
\end{eqnarray*}
(in particular $h$ is bounded on $\C$).
It follows by Lemma~\ref{deltagen} that for all
$p,\,q\geq0$, 
\begin{eqnarray*}
\left|
\Delta_{p,(\eta_p,\ldots,\eta_1)}
\left[
\left(
\frac{\overline{\zeta}}{1+|\zeta|^2}
\right)^q
\right]
\left(\eta_{p+1}\right)
\right|
& \leq &
\left\|
\left(
\frac{\overline{\zeta}}{1+|\zeta|^2}
\right)^q
\right\|_{\infty}
\times
R_{\eta}^p
\;\leq\;
R_{\eta}^p
\;\leq\;
R_{\eta}^{p+q}
\,,
\end{eqnarray*}
i.e. $\left(\eta_j\right)_{j\geq1}$ satisfies 
condition~(\ref{criter}) from Theorem~\ref{theorem1}
and this completes the proof of Proposition~\ref{propctrex}.

\end{proof}

\bigskip

\section{An essential result on the extraction of subsequences}\label{subsequaux}

In this part we give the proof of Proposition~\ref{norealint}
that will be useful in order to prove Theorem~\ref{fusion}. 
We first need a couple of preliminar results.

\subsection{Some reminders and preliminar results}

Let fix $\eta=\left(\eta_j\right)_{j\geq1}$.
We first remind the following identity that is justified by
Proposition~3
from~\cite{amadeo4} and that involves 
another analogous formula $R_N(\cdot;\eta)$,
that is the essential remainder part of $f-E_N(f;\eta)$:
$f\in\mathcal{O}\left(B_2(0,r_0)\right)$
(resp. $f\in\mathcal{O}\left(\C^2\right)$) being given,
one has
for all $N\geq1$ and all $z\in B_2(0,r_0)$ (resp. $z\in\C^2$),
\begin{eqnarray}\label{relapp}
f(z) & = &
E_N(f;\eta)(z)
-R_N(f;\eta)(z)
+\sum_{k+l\geq N}
a_{k,l}z_1^kz_2^l
\,;
\end{eqnarray}
here,
\begin{eqnarray*}
f(z)
& = &
\sum_{k,l\geq0}a_{k,l}z_1^kz_2^l
\end{eqnarray*}
is the Taylor expansion of $f$, and
\begin{eqnarray}\label{defRN}
& &
R_N(f;\eta)(z):=
\sum_{p=1}^N
\left(
\prod_{j=1,j\neq p}^N\frac{z_1-\eta_jz_2}{\eta_p-\eta_j}
\right)
\sum_{k+l\geq N}a_{k,l}\eta_p^k
\left(
\frac{z_2+\overline{\eta_p}z_1}{1+|\eta_p|^2}
\right)^{k+l-N+1}
\end{eqnarray}
is well-defined and belongs to
$\mathcal{O}\left(B_2(0,r_0)\right)$
(resp. $\mathcal{O}\left(\C^2\right)$).

Since the Taylor expansion of $f\in\mathcal{O}\left(B_2(0,r_0)\right)$
(resp. $f\in\mathcal{O}\left(\C^2\right)$) always converges to $0$
uniformly on any compact subset $K\subset B_2\left(0,r_0\right)$
(resp. $K\subset\C^2$), this gives an equivalence between the convergence
of $E_N(\cdot;\eta)$ and the one of $R_N(\cdot;\eta)$. More precisely,
we have the following result that is Lemma~7 from~\cite{amadeo4}.

\begin{lemma}\label{ENRN}

$r_0>0$ being fixed, let consider $f\in\mathcal{O}\left(B_2(0,r_0)\right)$
(resp. $f\in\mathcal{O}\left(\C^2\right)$) and $K$ any compact subset
of $B_2(0,r_0)$ (resp. $\C^2$). Then for all $N\geq1$, one has
\begin{eqnarray*}
\sup_{z\in K}
\left|
f(z)-E_N(f;\eta)(z)
\right|
& \leq &
\sup_{z\in K}
\left|
R_N(f;\eta)(z)
\right|
+
C_K(N+2)
\sup_{\|z\|\leq r_K}|f(z)|\;
(1-\varepsilon_K)^N
\,,
\end{eqnarray*}
where $\|z\|=\sqrt{|z_1|^2+|z_2|^2}$ is the usual norm on $\C^2$ and
$C_K$, $r_K$ depend only on $K$.


In particular, $E_N(f;\eta)$ converges to $f$
(uniformly on any compact subset) if and only if so does
$R_N(f;\eta)$ to $0$.

\end{lemma}

On the other hand, we will also deal with the action of some homographic transformations
on $\left(\eta_j\right)_{j\geq1}$. We remind
some notations and results from~\cite{amadeo4} (beginning of Section~4): let fix any
$\eta^c\notin\left\{\eta_j\right\}_{j\geq1}\cup\{\infty\}$ and let consider the unitary matrix
$U_{\eta^c}\in\mathcal{U}(2,\C)$ defined by
\begin{eqnarray*}
U_{\eta^c}
& := &
\frac{1}{\sqrt{1+|\eta^c|^2}}
\bordermatrix{ & & \cr
& \overline{\eta^c}  & 1  \cr
& 1  & -\eta^c  \cr}
\,;
\end{eqnarray*}
let also consider the following homographic application
\begin{eqnarray}\label{homog}
h_{\eta^c}\;:\;\C{\bf P}^1
& \rightarrow &
\C{\bf P}^1
\\\nonumber
\zeta & \mapsto &
\frac{1+\overline{\eta^c}\,\zeta}{\zeta-\eta^c}
\end{eqnarray}
(where $\C{\bf P}^1=\C\cup\{\infty\}$) and the new sequence
\begin{eqnarray}\label{defthetaj}
\theta
\;=\;
\left(\theta_j\right)_{j\geq1}
& := &
\left(h_{\eta^c}\left(\eta_j\right)\right)_{j\geq1}
\,.
\end{eqnarray}
Then the set $\left\{\theta_j\right\}_{j\geq1}$ is well-defined as a subset of $\C$
and one has the following result that is Lemma~16 from~\cite{amadeo4}.

\begin{lemma}\label{lemmextensio}

Let be $f\in\mathcal{O}\left(B_2(0,r_0)\right)$
(resp. $f\in\mathcal{O}\left(\C^2\right)$).
For all $N\geq1$ and $z\in B_2(0,r_0)$ (resp. $z\in\C^2$),
\begin{eqnarray*}
R_N(f;\eta)(z) & = &
R_N\left(f\circ U_{\eta^c}^{-1};\theta\right)
\left(U_{\eta^c}z\right)
\,.
\end{eqnarray*}

\end{lemma}

These lemmas yield the following consequence.

\begin{corollary}\label{lemmextensioENRN}

$\eta=\left(\eta_j\right)_{j\geq1}$ being any sequence,
$\eta^c\notin\left\{\eta_j\right\}_{j\geq1}\cup\{\infty\}$ being fixed
and $h_{\eta^c}$ (resp. $\theta$) being defined by~(\ref{homog})
(resp.~(\ref{defthetaj})),
the formula $E_N(f;\eta)$ converges to $f$
uniformly on any compact subset $K\subset\C^2$ and for every
function $f\in\mathcal{O}\left(\C^2\right)$,
if and only if so does $E_N(f;\theta)$.

\end{corollary}

\begin{proof}

$f\in\mathcal{O}\left(\C^2\right)$ being given,
one has by Lemma~\ref{ENRN} that the formula
$E_N(f;\eta)$ converges to $f$ if and only if
$R_N(f;\eta)$ converges to $0$ (uniformly on any compact subset).
$U_{\eta^c}$ being an isometry, it follows by Lemma~\ref{lemmextensio} that
$R_N(f;\eta)$ converges to $0$
uniformly on any compact subset $K\subset\C^2$ and for every
function $f\in\mathcal{O}\left(\C^2\right)$,
if and only if so does $R_N\left(f\circ U_{\eta^c}^{-1};\theta\right)$,
thus if and only if so does $R_N(f;\theta)$ for
all $f\in\mathcal{O}\left(\C^2\right)$.
Finally, by applying Lemma~\ref{ENRN} again,
it is true if and only if
$E_N\left(f;\theta\right)$ converges to $f$
(uniformly on any compact subset) for all $f\in\mathcal{O}\left(\C^2\right)$.

\end{proof}

We also prove the following preliminar result about
the homographic transformations defined by~(\ref{homog}).

\begin{lemma}\label{hu-1}

For all $\eta^c\notin\left\{\eta_j\right\}_{j\geq1}\cup\{\infty\}$, one has 
$h_{\eta^c}^{-1}=h_{\overline{\eta^c}}$ where
\begin{eqnarray*}
h_{\overline{{\eta^c}}}\;:\;\C{\bf P}^1
& \rightarrow &
\C{\bf P}^1
\\\nonumber
\zeta & \mapsto &
\frac{1+{\eta^c}\,\zeta}{\zeta-\overline{\eta^c}}
\,.
\end{eqnarray*}
In addition, one also has that
$\overline{\eta^c}\notin\left\{h_{\eta^c}\left(\eta_j\right)\right\}_{j\geq1}\bigcup\{\infty\}$,
i.e. $h_{\eta^c}^{-1}=h_{\overline{\eta^c}}$  is of the same kind~(\ref{homog})
for the associated set 
$\left\{h_{\eta^c}\left(\eta_j\right)\right\}_{j\geq1}=\left\{\theta_j\right\}_{j\geq1}$.

\end{lemma}

\begin{proof}

Indeed, for all 
$\zeta\in\C\setminus\left\{\overline{\eta^c}\right\}$, one has that
\begin{eqnarray*}
\left(h_{\eta^c}\circ h_{\overline{\eta^c}}\right)(\zeta)
& = &
\frac{
1+\overline{\eta^c}\,
\dfrac{
1+{\eta^c}\,\zeta
}{
\zeta-\overline{\eta^c}
}
}{
\dfrac{
1+{\eta^c}\,\zeta
}{
\zeta-\overline{\eta^c}
}
-\eta^c
}
\;=\;
\frac{
\zeta+\eta^c\overline{\eta^c}\,\zeta
}{
1+\eta^c\overline{\eta^c}
}
\;=\;
\zeta
\,,
\end{eqnarray*}
then the equality holds for all $\zeta\in\C{\bf P}^1$.
The second assertion follows by~(\ref{homog}) since $h_{\eta^c}(\infty)=\overline{\eta^c}$,
then $h_{\eta^c}\left(\eta_j\right)\neq\overline{\eta^c}$ for all $j\geq1$.

\end{proof}

We finish the subsection with the following result
reminded as Lemma~18
from~\cite{amadeo4}, and  that gives an equivalent definition for
the geometric criterion~(\ref{geom}).

\begin{lemma}\label{curve}

The set
$\{\eta_j\}_{j\geq1}$ is locally interpolable
by real-analytic curves 
if and only if it can locally holomorphically interpolate 
the conjugate function, i.e. for all
$\zeta\in\overline{\{\eta_j\}_{j\geq1}}$
(the topological closure of $\{\eta_j\}_{j\geq1}$ in $\C{\bf P}^1$),
there are a neighborhood $V$ of $\zeta$
and $g\in\mathcal{O}\left(V\right)$ such that
\begin{eqnarray}\label{defequivanal}
\overline{\eta_j}
& = &
g(\eta_j)\,,\;
\forall\,\eta_j\in V
\,.
\end{eqnarray}

\end{lemma}

\medskip

\subsection{On the extraction of certain subsequences}\label{subsequnotrealanal}

Now we can give the proof of Proposition~\ref{norealint}.

\begin{proof}

Since the set $\left\{\eta_j\right\}_{j\geq1}$ is not locally interpolable by real-analytic curves,
it follows by Lemma~\ref{curve} that there is 
$\zeta_0\in\overline{\left\{\eta_j\right\}_{j\geq1}}$
without any neighborhood 
$V\in\mathcal{V}\left(\zeta_0\right)$
and holomorphic function
$g\in\mathcal{O}\left(V_{\zeta_0}\right)$
that can interpolate the conjugate function on
$\left\{\eta_j\right\}_{j\geq1}\bigcap V$, i.e.
\begin{eqnarray}\label{notanalint}
\forall\,V\in\mathcal{V}\left(\zeta_0\right),\,
\forall\,g\in\mathcal{O}(V),\,
\exists\,\eta_j\in V,\;
g\left(\eta_j\right)
\;\neq\;
\overline{\eta_j}
\,.
\end{eqnarray}
In particular, $\zeta_0$ cannot be isolated in
$\left\{\eta_j\right\}_{j\geq1}$.
Otherwise, if $\zeta_0\neq\infty$
(resp. $\zeta_0=\infty$), then by taking
$V_{\zeta_0}\subset\C$ such that
$\left\{\eta_j\right\}_{j\geq1}\bigcap V_{\zeta_0}=\left\{\zeta_0\right\}$
(resp. $V_{\infty}=\C{\bf P}^1\setminus K$, where the compact subset
$K$ is big enough so that
$\left\{\eta_j\right\}_{j\geq1}\setminus K\subset\{\infty\}$)
and 
\begin{eqnarray*}
g_{\zeta_0}
\;:\;
V_{\zeta_0}
& \rightarrow &
\C
\\
\zeta
& \mapsto &
g_{\zeta_0}(\zeta)\equiv\overline{\zeta_0}
\end{eqnarray*}
(resp.
\begin{eqnarray*}
g_{\infty}
\;:\;
V_{\infty}
& \rightarrow &
\C{\bf P}^1
\\
\zeta
& \mapsto &
g_{\infty}(\zeta)\equiv\infty
\;)
\,,
\end{eqnarray*}
we would get a contradiction with~(\ref{notanalint}).

As a consequence, there is a  subsequence
$\left(\eta_{j_k}\right)_{k\geq1}\subset\left(\eta_j\right)_{j\geq1}$ that satisfies:
\begin{eqnarray}\label{etajk}
\begin{cases}
\mbox{$\left(\eta_{j_k}\right)_{k\geq1}$ converges to $\zeta_0$;}
\\
\mbox{$\eta_{j_k}\neq\zeta_0$ for all $k\geq1$.}
\end{cases}
\end{eqnarray}
We will then deal with the cases $\zeta_0\in\C$ and $\zeta_0=\infty$
respectively.

\medskip

\subsubsection*{\underline{$\zeta_0\in\C$}}
We start by setting
$S_0
:=
\left(\eta_{j_k}\right)_{k\geq1}$
and
\begin{eqnarray}\label{defS1}
S_1
& := &
S_0\bigcap D\left(\zeta_0,1\right)
\,,
\end{eqnarray}
where $D\left(\zeta_0,1\right)=\left\{\zeta\in\C,\,\left|\zeta-\zeta_0\right|<1\right\}$.
By construction, $S_1$ gives a (nonempty and infinite) sequence
that converges to $\zeta_0$.
If $S_1$ (as a set) is not locally interpolable by real-analytic curves, the proposition is proved.
Otherwise (because $\zeta_0$ is a limit point of $S_1$), there are
$V_1\in\mathcal{V}\left(\zeta_0\right)$
and
$g_1\in\mathcal{O}\left(V_1\right)$
such that
\begin{eqnarray*}
\overline{\eta_j}
& = &
g_1\left(\eta_j\right)
\mbox{ for all }\eta_j\in S_1\cap V_1
\,.
\end{eqnarray*}
By reducing $V_1$ if necessary, we can assume that
$V_1\subset D\left(\zeta_0,1\right)$ and 
$V_1$ is connected.
Since $\zeta_0$ satisfies~(\ref{notanalint}), it follows that
$g_1|_{V_1}$ cannot interpolate the conjugate function
on $V_1\bigcap\{\eta_j\}_{j\geq1}$, i.e.
there is $\eta_{s_1}\in V_1\bigcap\{\eta_j\}_{j\geq1}$
such that $g_1\left(\eta_{s_1}\right)\neq\overline{\eta_{s_1}}$.
We set 
\begin{eqnarray*}
S_2
& := &
S_1
\bigcup
\left\{\eta_{s_1}\right\}
\end{eqnarray*}
and
$S_2$ (with any enumeration) still gives a sequence that converges to $\zeta_0$.

Let fix $m\geq1$ and let assume having constructed
$\eta_{s_1},\ldots,\eta_{s_m}$,
$S_1,\ldots,S_m$,
$V_1,\ldots,V_m$ and $g_1,\ldots,g_m$
such that for all
$q=1,\ldots,m$, one has the following properties:
\begin{eqnarray}\label{inducnorealint1}
& &
V_q\in\mathcal{V}\left(\zeta_0\right)
\mbox{ and }
V_q\mbox{ is connected; }
\;\;\;\;\;\;\;\;\;\;\;\;\;\;\;\;\;\;\;\;\;\;\;\;\;\;\;\;\;\;\;\;
\;\;\;\;\;\;\;\;\;\;\;\;\;\;\;\;\;\;\;\;\;\;\;\;\;\;\;
\\\label{inducnorealint2}
& &
\eta_{s_q}\in V_q\subset D\left(\zeta_0,1/2^{q-1}\right)
\,;
\\\label{inducnorealint3}
& &
g_q\in\mathcal{O}\left(V_q\right)
\mbox{ and }
g_q\left(\eta_{s_q}\right)\neq\overline{\eta_{s_q}}
\,;
\\\label{inducnorealint4}
& &
g_q\left(\eta_j\right)=\overline{\eta_j}
\;\mbox{ for all }
\eta_j\in S_{q}\cap V_q
\,,
\end{eqnarray}
where
\begin{eqnarray}\label{defSm}
S_{q}
& = &
S_1\bigcup\left\{\eta_{s_1},\ldots,\eta_{s_{q-1}}\right\}
\mbox{ for all }
q=2,\ldots,m
\,.
\end{eqnarray}

We first consider the set
\begin{eqnarray*}
S_{m+1}
& := &
S_1\bigcup\left\{\eta_{s_1},\ldots,\eta_{s_m}\right\}
\end{eqnarray*}
and this satisfies~(\ref{defSm}) for all $q=2,\ldots,m+1$.
Next, $S_{m+1}$ (with any enumeration) will give a sequence that still converges to $\zeta_0$ as
the union of $S_1$ (that converges to $\zeta_0$ by~(\ref{defS1}) and~(\ref{etajk}))
and the finite set 
$\left\{\eta_{s_1},\ldots,\eta_{s_m}\right\}$.
If $S_{m+1}$ is not locally interpolable by real-analytic curves,
the proposition is proved.
Otherwise (because $\zeta_0$ is a limit point of $S_{m+1}$), 
there are $V_{m+1}\in\mathcal{V}\left(\zeta_0\right)$
and $g_{m+1}\in\mathcal{O}\left(V_{m+1}\right)$ such that
\begin{eqnarray*}
\overline{\eta_j}
& = &
g_{m+1}\left(\eta_j\right)
\mbox{ for all }
\eta_j\in S_{m+1}\cap V_{m+1}
\,.
\end{eqnarray*}
%
By reducing $V_{m+1}$ if necessary, we can assume that
$V_{m+1}\subset D\left(\zeta_0,1/2^m\right)$ and 
$V_{m+1}$ is connected.
On the other hand, since $\zeta_0$ satisfies~(\ref{notanalint}), 
it follows that there is 
$\eta_{s_{m+1}}\in V_{m+1}$ such that
$g_{m+1}\left(\eta_{s_{m+1}}\right)\neq\overline{\eta_{s_{m+1}}}$.
This proves~(\ref{inducnorealint1}), (\ref{inducnorealint2}), (\ref{inducnorealint3})
and~(\ref{inducnorealint4}) for $q=m+1$, and completes the induction.

Now if there is $m\geq1$ such that the set
$S_{m}$ defined by~(\ref{defSm}) is not locally interpolable by real-analytic curves,
then the proposition is proved (since any enumeration of $S_{m}$ will give a sequence that
converges to $\zeta_0$).

Otherwise, we can construct for all $m\geq1$, such 
$\eta_{s_m}$, $S_m$, $V_m$ and $g_m$ that 
fulfill~(\ref{inducnorealint1}), (\ref{inducnorealint2}), (\ref{inducnorealint3})
and~(\ref{inducnorealint4})
for all $q=1,\ldots,m$, and consider the following set
\begin{eqnarray*}
S_{\infty}
& := &
\bigcup_{m\geq1}S_{m}
\;=\;
S_1
\bigcup
\left\{
\eta_{s_m},\,m\geq1
\right\}
\,.
\end{eqnarray*}
Then any enumeration of $S_{\infty}$ will give a sequence that converges to $\zeta_0$
as the union of $S_1$ (that converges to $\zeta_0$ by~(\ref{defS1}) and~(\ref{etajk}))
and the convergent sequence
$\left(\eta_{s_m}\right)_{m\geq1}$ (since by~(\ref{inducnorealint2}),
one has $\eta_{s_m}\in D\left(\zeta_0,1/2^{m-1}\right)$). 
If we prove that $S_{\infty}$ is not locally interpolable by
real-analytic curves, the proof of the proposition will be achieved
in the case for which $\zeta_0\in\C$.

Let assume on the contrary that $S_{\infty}$ is, i.e.
(since $\zeta_0$ is a limit point of $S_{\infty}$)
there are 
$V_{\infty}\in\mathcal{V}\left(\zeta_0\right)$ and
$g_{\infty}\in\mathcal{O}\left(V_{\infty}\right)$ such that
$g_{\infty}\left(\eta_j\right)
=\overline{\eta_j}$ for all
$\eta_j\in S_{\infty}\bigcap V_{\infty}$.
In particular, this yields for all $m\geq1$ (since
$S_m\subset S_{\infty}$),
\begin{eqnarray}\label{ginfty}
g_{\infty}\left(\eta_j\right)
& = &
\overline{\eta_j}
\,,
\;\;\;\forall\;
\eta_j\in S_{m}\bigcap V_{\infty}
\,.
\end{eqnarray}
On the other hand, by~(\ref{inducnorealint2}) there is
$m_0\geq1$ such that
$V_{m_0}\subset D\left(\zeta_0,1/2^{m_0-1}\right)
\subset V_{\infty}$.
In addition, one has
by~(\ref{inducnorealint4}) for $q=m_0$, that
\begin{eqnarray*}
g_{m_0}\left(\eta_j\right)
& = &
\overline{\eta_j}
\;\mbox{ for all }
\eta_j\in S_{{m_0}}\cap V_{m_0}
\,.
\end{eqnarray*}
Hence
$g_{m_0}$ and $g_{\infty}|_{V_{m_0}}$ are both holomorphic functions
on the domain $V_{m_0}$, that coincide on the set
$S_{{m_0}}\cap V_{m_0}$.
Since by~(\ref{defSm}), $S_{{m_0}}\cap V_{m_0}\supset S_1\cap V_{m_0}$
that is infinite with limit point $\zeta_0\in V_{m_0}$
by~(\ref{defS1}), (\ref{etajk}) and~(\ref{inducnorealint1}),
it follows that 
\begin{eqnarray}\label{ginfty=gm0}
g_{\infty}|_{V_{m_0}}
& \equiv &
g_{m_0}
\,.
\end{eqnarray}
But an application of~(\ref{ginfty}) for $m=m_0+1$ yields
$g_{\infty}\left(\eta_j\right)=\overline{\eta_j}$
for all $\eta_j\in S_{m_0+1}\cap V_{\infty}$.
In particular, since
$\eta_{s_{m_0}}\in V_{m_0}\subset D\left(\zeta_0,1/2^{m_0-1}\right)\subset V_{\infty}$ 
 (by~(\ref{inducnorealint2}) for $q=m_0$)
and $\eta_{s_{m_0}}\in S_{m_0+1}$ by~(\ref{defSm}) for $m=m_0+1$,
one has $\eta_{s_{m_0}}\in S_{m_0+1}\cap V_{\infty}$ then
\begin{eqnarray}\label{ginftyinterpol}
g_{\infty}\left(\eta_{s_{m_0}}\right)
& = &
\overline{\eta_{s_{m_0}}}
\,.
\end{eqnarray}
Moreover, an application of~(\ref{inducnorealint3}) for $q=m_0$, also
yields
\begin{eqnarray}\label{gm0ninterpol}
g_{m_0}\left(\eta_{s_{m_0}}\right)
\neq
\overline{\eta_{s_{m_0}}}
\,.
\end{eqnarray}
Finally, (\ref{ginfty=gm0}), (\ref{ginftyinterpol}) and~(\ref{gm0ninterpol})
together lead to (since $\eta_{s_{m_0}}\in V_{m_0}$)
\begin{eqnarray*}
\overline{\eta_{s_{m_0}}}
\;=\;
g_{\infty}\left(\eta_{s_{m_0}}\right)
\;=\;
g_{\infty}|_{V_{m_0}}\left(\eta_{s_{m_0}}\right)
\;=\;
g_{m_0}\left(\eta_{s_{m_0}}\right)
\;\neq\;
\overline{\eta_{s_{m_0}}}
\,,
\end{eqnarray*}
and this is impossible. Necessarily, $S_{\infty}$ cannot be
locally interpolable by real-analytic curves and the proposition is proved
in the case for which $\zeta_0\in\C$.

\medskip

\subsubsection*{\underline{$\zeta_0=\infty$}}
First, by
removing $0$ from $\left\{\eta_j\right\}_{j\geq1}$ if necessary,
we can assume that $\eta_j\neq0,\,\forall\,j\geq1$
(as well as $\eta_{j_k}\neq0,\,\forall\,k\geq1$).
Indeed, since the sequence
$\left(\eta_{j_k}\right)_{k\geq1}$ converges to $\infty$
by~(\ref{etajk}),
it follows that the subset $\left\{\eta_{j_k}\right\}_{k\geq1}\setminus\{0\}$
is infinite, then so is the set
$\left\{\eta_j\right\}_{j\geq1}\setminus\{0\}
\supset\left\{\eta_{j_k}\right\}_{k\geq1}\setminus\{0\}$.
In addition, the new subset
$\left\{\eta_{j_k}\right\}_{k\geq1}\setminus\{0\}$ gives a new
sequence that still satisfies~(\ref{etajk}).

Now let consider the sequence
$\left(\theta_j\right)_{j\geq1}$ where
\begin{eqnarray*}
\theta_j
\;:=\;
\frac{1}{\eta_j}
\;\mbox{ for all }\;
j\geq1
\,.
\end{eqnarray*}
First, $\left(\theta_j\right)_{j\geq1}$ is well-defined.
Next, since $\left(\eta_{j_k}\right)_{k\geq1}$
satisfies~(\ref{etajk}) with $\zeta_0=\infty$, it follows that
so does the subsequence
$\left(\theta_{j_k}\right)_{k\geq1}$ with the choice of $\zeta_0':=0$, i.e.
\begin{eqnarray}\label{thetajk}
\begin{cases}
\mbox{$\left(\theta_{j_k}\right)_{k\geq1}$ converges to $0$\,;}
\\
\mbox{$\theta_{j_k}\neq0$ for all $k\geq1$.}
\end{cases}
\end{eqnarray}

Lastly, we claim that $\zeta_0'=0$ satisfies~(\ref{notanalint}) as well. Indeed, let be
$V\in\mathcal{V}(0)$ and $g\in\mathcal{O}\left(V\right)$. We want to prove that
there exists $\theta_j\in V$ such that
$g\left(\theta_j\right)\neq\overline{\theta_j}$.

If $g(0)\neq0$, then by~(\ref{thetajk}),
$\overline{\theta_{j_k}}\longrightarrow0$
and $g\left(\theta_{j_k}\right)\longrightarrow g(0)\neq0$
as $k\rightarrow+\infty$. It follows that
$g\left(\theta_{j_k}\right)\neq\overline{\theta_{j_k}}$ for all
$k$ large enough and the claim is proved in this case.

Otherwise, $g(0)=0$. 
Let consider 
\begin{eqnarray*}
W 
&:= &
\left\{\frac{1}{\zeta}\,,\;\zeta\in V\setminus\{0\}\right\}
\bigcup\,
\{\infty\}
\end{eqnarray*}
and
\begin{eqnarray*}
h\;:\;W
& \longrightarrow &
\C{\bf P}^1
\\
\infty
& \longmapsto &
\infty
\,,
\\
\zeta\in\C\cap W 
& \longmapsto &
\begin{cases}
\dfrac{1}{g(1/\zeta)}
\;\mbox{ if }\;
g(1/\zeta)\neq0
\,,
\\
\infty
\;\mbox{ otherwise.}
\end{cases}
\end{eqnarray*}
Then $W\in\mathcal{V}(\infty)$, $h$ is well-defined
and $h\in\mathcal{O}(W)$. It follows by~(\ref{notanalint})
that there is $\eta_j\in W$ such that
$h\left(\eta_j\right)\neq\overline{\eta_j}$.
If $g(1/\eta_j)=0$, i.e.
$g\left(\theta_j\right)=0$, then
$\overline{\theta_j}=1/\overline{\eta_j}\neq0=g\left(\theta_j\right)$,
and this proves the claim in that case.
Otherwise, $g\left(1/\eta_j\right)\neq0$, i.e. $g\left(\theta_j\right)\neq0$ then
\begin{eqnarray*}
\frac{1}{g\left(\theta_j\right)}
\;=\;
\frac{1}{g\left(1/\eta_j\right)}
\;=\;
h\left(\eta_j\right)
& \neq &
\overline{\eta_j}
\;=\;
\frac{1}{\overline{\theta_j}}
\,,
\end{eqnarray*}
hence
$g\left(\theta_j\right)\neq\overline{\theta_j}$
and the claim is proved in this last case.

We can now apply the previous case of the proposition with the choice of 
$\left(\theta_j\right)_{j\geq1}$ and $\zeta_0'=0$ to get a subsequence
$\left(\theta_{j'_k}\right)_{k\geq1}$ (maybe different from
$\left(\theta_{j_k}\right)_{k\geq1}$) that converges to $0$ and
that is not locally interpolable by real-analytic curves.
It follows that the sequence
$\left(\eta_{j'_k}\right)_{k\geq1}=\left(1/\theta_{j'_k}\right)_{k\geq1}$
converges to $\infty$. On the other hand,
the inverse function $\zeta\mapsto1/\zeta$ being a homographic transformation,
it is in particular a biholomorphic application of $\C{\bf P}^1$. Hence the subset
$\left\{\eta_{j'_k}\right\}_{k\geq1}=\left\{1/\theta_{j'_k}\right\}_{k\geq1}$ 
cannot be either locally interpolable by
real-analytic curves. 
Finally, one also has that $\eta_{j'_k}\neq0$ for all $k\geq1$.
This proves the proposition in this second case
and completes its whole proof.

\end{proof}

\begin{remark}\label{etajkn0}

As we have seen in the above proof, we know in addition that 
in the case for which $\zeta_0=\infty$,
we also have that $\eta_{j_k}\neq0$ for all $k\geq1$.

\end{remark}

\bigskip

\section{Proof of the main theorem}

In the first subsection, we deal with the proof of the equivalence 
between~(\ref{fusiongeom}) and~(\ref{fusionsubseq}) in the statement 
of Theorem~\ref{fusion}.

\medskip

\subsection{On the stability by extraction of subsequences}

Before giving the proof of this part,
we remind the following result as Proposition~2 from~\cite{amadeo4},
and that is a special case of equivalence for
the geometric criterion~(\ref{geom}), i.e.
in the particular case when
$\left(\eta_j\right)_{j\geq1}$ is a convergent sequence, 
condition~(\ref{geom}) also becomes necessary.

\begin{proposition}\label{specialequiv}

Let $\left(\eta_j\right)_{j\geq1}$ be any convergent sequence (in $\C$).
If the interpolation formula
$E_N(f;\eta)$ converges to $f$ (uniformly on any compact subset)
for all $f\in\mathcal{O}\left(\C^2\right)$, then
$\{\eta_j\}_{j\geq1}$ is locally interpolable by real-analytic curves.

\end{proposition}

\begin{remark}\label{remark5.2}

To be rigorous, in order to apply Proposition~\ref{specialequiv},
we should also assume that $E_N(f;\eta)$ converges to $f$
for all $f\in\mathcal{O}\left(B_2(0,r_0)\right)$. But as specified by 
Remark~5.2 from~\cite{amadeo4}, it is sufficient to assume the convergence
of $E_N(f;\eta)$ for all $f\in\mathcal{O}\left(\C^2\right)$.

\end{remark}

\begin{proof}

First, if $\left\{\eta_j\right\}_{j\geq1}$ is
locally interpolable by real-analytic curves,
then so is any (infinite) subset
$\left\{\eta_{j_k}\right\}_{k\geq1}$.
The  implication~(\ref{fusiongeom})$\implies$(\ref{fusionsubseq}) then
follows by Theorem~\ref{theorem2}.

Conversely, let assume that
$\left\{\eta_j\right\}_{j\geq1}$ is not
locally interpolable by real-analytic curves.
By Proposition~\ref{norealint},
there is a subsequence
$\left(\eta_{j_k}\right)_{k\geq1}$ 
that is not locally interpolable by real-analytic curves
and that is convergent (in $\C{\bf P}^1$).
In order to get the converse implication~(\ref{fusionsubseq})$\implies$(\ref{fusiongeom}),
we want to prove that
$\eta':=\left(\eta_{j_k}\right)_{k\geq1}$ 
does not make converge its associated interpolation formula
$E_N\left(\cdot;\eta'\right)$ for entire functions, i.e.
there exists (at least) one function $f\in\mathcal{O}\left(\C^2\right)$
such that $E_N\left(f;\eta'\right)$ does not converge to $f$
(uniformly on any compact subset $K\subset\C^2$).

Let be $\zeta_0=\lim_{k\rightarrow+\infty}\eta_{j_k}$.
If $\zeta_{0}$ is finite, the required assertion follows by Proposition~\ref{specialequiv}.
Otherwise, $\zeta_0=\infty$ and by Remark~\ref{etajkn0},
one also has that
$\eta_{j_k}\neq0$ for all $k\geq1$. It follows that the sequence
$\theta':=\left(\theta_{j_k}\right)_{k\geq1}$ where
\begin{eqnarray*}
\theta_{j_k}
& := &
\frac{1}{\eta_{j_k}}
\;
\mbox{ for all }
k\geq1
\,,
\end{eqnarray*}
is well-defined (as a subset of $\C$), bounded and converges to $0$.
On the other hand, $\theta_{j_k}=h_0\left(\eta_{j_k}\right)$
for all $k\geq1$, where $h_0$ is the homographic transformation
defined as $h_0(\zeta)=1/\zeta$ (see~(\ref{homog}) with the choice of $\eta^c=0$).
Thus $\left\{\theta_{j_k}\right\}_{k\geq1}$ is not locally interpolable by 
real-analytic curves (because any homographic transformation
is in particular biholomorphic).
Again, by Proposition~\ref{specialequiv}, the sequence
$\theta'=\left(\theta_{j_k}\right)_{k\geq1}$ does not make converge its
associated interpolation formula
$E_N\left(\cdot;\theta'\right)$ for entire functions, i.e.
there exists $f\in\mathcal{O}\left(\C^2\right)$ such that
$E_N\left(f;\theta'\right)$ does not converge to $f$ (uniformly on any compact subset).
Finally, since $E_N\left(\cdot;\theta'\right)=E_N\left(\cdot;h_0\left(\eta'\right)\right)$,
it follows by Corollary~\ref{lemmextensioENRN} that neither can do
the sequence
$\eta'=\left(\eta_{j_k}\right)_{k\geq1}$ 
for the formula $E_N\left(\cdot;\eta'\right)$ for entire functions, and this proves
the implication~(\ref{fusionsubseq})$\implies$(\ref{fusiongeom}).

\end{proof}

\medskip

\subsection{On the action by permutations}

Now we can give the proof of the second part of Theorem~\ref{fusion}
that is the equivalence between~(\ref{fusiongeom}) and~(\ref{fusionpermut}),
and achieve its whole proof. We first need a specific result that is
a part of the proof for Theorem~1 from~\cite{amadeo4}
(reminded above as Theorem~\ref{theorem1}).

\begin{lemma}\label{criternecessary}

Let be $\left(\eta_j\right)_{j\geq1}$ such that, for all
$f\in\mathcal{O}\left(\C^2\right)$,
$R_N(f;\eta)$ is uniformly bounded on any compact
subset of $\C^2$. Then the estimate~(\ref{criter}) from
Theorem~\ref{theorem1} is satisfied.

\end{lemma}

This result is Lemma~11 from~\cite{amadeo4} and yields the part (\ref{thm12})$\implies$(\ref{thm13})
in the statement of Theorem~\ref{theorem1}. In particular, the important fact
is that no one condition is needed for the set $\left\{\eta_j\right\}_{j\geq1}$ (like boundedness,
see Remark~3.1 from~\cite{amadeo4}).
This will be useful in order to prove the
implication~(\ref{fusionpermut})$\implies$(\ref{fusiongeom})
in the statement of Theorem~\ref{fusion}.

\begin{proof}

The implication~(\ref{fusiongeom})$\implies$(\ref{fusionpermut})
immediately follows by Theorem~\ref{theorem2} since
the property of being locally interpolable by
real-analytic curves is a condition about sets, then
it does not depend on any enumeration of
$\{\eta_j\}_{j\geq1}$.

\medskip

Conversely, let assume that $\{\eta_j\}_{j\geq1}$
(coming from the sequence $\eta=\left(\eta_j\right)_{j\geq1}$) is not locally interpolable by
real-analytic curves. We want to find a permutation $\sigma$ of
$\N\setminus\{0\}$ such that
$E_N\left(\cdot;\sigma(\eta)\right)$ does not converge
for entire functions (where $\sigma(\eta):=\left(\eta_{\sigma(j)}\right)_{j\geq1}$).
We know by Proposition~\ref{norealint} that there are
$\zeta_0\in\C{\bf P}^1$ and
a subsequence $\eta'=\left(\eta_{j_k}\right)_{k\geq1}$ of
$\left(\eta_j\right)_{j\geq1}$ that satisfy
the following conditions:
\begin{eqnarray}\label{remindnorealint}
& &
\begin{cases}
\mbox{the sequence $\left(\eta_{j_k}\right)_{k\geq1}$
converges to $\zeta_0$;}
\\
\mbox{the set $\left\{\eta_{j_k}\right\}_{k\geq1}$
is not locally interpolable by real-analytic curves.}
\end{cases}
\end{eqnarray}

First, we claim that can w.l.o.g. assume that $\zeta_0$ is finite. Indeed, if $\zeta_0=\infty$, 
let consider 
$\eta^c\notin\left\{\eta_j\right\}_{j\geq1}\bigcup\{\infty\}$,
$h_{\eta^c}\in\mathcal{O}\left(\C{\bf P}^1\right)$ defined by~(\ref{homog})
and the associated sequence
$\theta=\left(\theta_j\right)_{j\geq1}=\left(h_{\eta^c}\left(\eta_j\right)\right)_{j\geq1}$ 
(that is well-defined by~(\ref{defthetaj})).
Then the subsequence
$\left(\theta_{j_k}\right)_{k\geq1}=\left(h_{\eta^c}\left(\eta_{j_k}\right)\right)_{k\geq1}$ 
satisfies~(\ref{remindnorealint})
with $\zeta_0'=h_{\eta^c}(\infty)=\overline{\eta^c}\in\C$
(because $\left(\eta_{j_k}\right)_{k\geq1}$ does and
$h_{\eta^c}$ is biholomorphic).
It will follow that there will be a permutation $\sigma$ such that
the sequence $\sigma(\theta)=\left(\theta_{\sigma(j)}\right)_{j\geq1}$
does not make converge its associated interpolation formula
$E_N\left(\cdot;\sigma(\theta)\right)$ for entire functions. 
Since 
\begin{eqnarray*}
h_{\eta^c}^{-1}\left(\sigma(\theta)\right)
& = &
h_{\eta^c}^{-1}\left[\left(\theta_{\sigma(j)}\right)_{j\geq1}\right]
\;=\;
\left(h_{\eta^c}^{-1}\left(\theta_{\sigma(j)}\right)\right)_{j\geq1}
\;=\;
\left(h_{\eta^c}^{-1}\left[h_{\eta^c}\left(\eta_{\sigma(j)}\right)\right]\right)_{j\geq1}
\\
& = &
\left(\eta_{\sigma(j)}\right)_{j\geq1}
\;=\;
\sigma(\eta)
\end{eqnarray*}
(notice that $\left\{\left(h_{\eta^c}^{-1}\left(\sigma(\theta)\right)\right)_j\right\}_{j\geq1}
=\left\{h_{\eta^c}^{-1}\left(\theta_j\right)\right\}_{j\geq1}$ is well-defined as a subset of $\C$
by Lemma~\ref{hu-1}),
an application
of Corollary~\ref{lemmextensioENRN} (which is possible because
$h_{\eta^c}^{-1}=h_{\overline{\eta^c}}$ by Lemma~\ref{hu-1}) will allow us to deduce that
neither will do the sequence $h_{\eta^c}^{-1}\left(\sigma(\theta)\right)=\sigma(\eta)$
for $E_N\left(\cdot;\sigma(\eta)\right)$ for entire functions, i.e.
there will exist (at least) one function $f\in\mathcal{O}\left(\C^2\right)$ such that
$E_N(f;\sigma(\eta))$ will not converge to $f$
(uniformly on any compact subset $K\subset\C^2$). This
will prove the required  
implication~(\ref{fusionpermut})$\implies$(\ref{fusiongeom})
of the theorem for the case $\zeta_0=\infty$
and complete the whole proof of the equivalence
between~(\ref{fusiongeom}) and~(\ref{fusionpermut}) in the general case.

\medskip

We can then assume that $\zeta_0\in\C$ in~(\ref{remindnorealint}). Let
fix the enumeration of the associated subsequence
$\eta'=\left(\eta_{j_k}\right)_{k\geq1}$ as well as the canonical one for the complementary subsequence
\begin{eqnarray}\label{defeta''}
\eta''
\;=\;
\left(\eta_{r_m}\right)_{m\geq1}
\;:=\;
\left(\eta_j\right)_{j\geq1}
\setminus
\left(\eta_{j_k}\right)_{k\geq1}
\,.
\end{eqnarray}
Since $\eta'=\left(\eta_{j_k}\right)_{k\geq1}$ is a (bounded) convergent sequence
that is not locally interpolable by real-analytic curves, it follows by Proposition~\ref{specialequiv}
that $E_N\left(\cdot;\eta'\right)$
cannot converge for entire functions. By an application 
of~(\ref{thm12})$\Longleftrightarrow$(\ref{thm13}) in
Theorem~\ref{theorem1}, it follows
that the sequence of the associated divided differences is not exponentially bounded,
i.e. $\forall\,R\geq1$, $\exists\,p_R,\,q_R\geq0$ such that
\begin{eqnarray}\label{notcriter}
\left|
\Delta_{p_R,(\eta_{j_{p_R}},\ldots,\eta_{j_1})}
\left[
\left(
\frac{\overline{\zeta}}{1+|\zeta|^2}
\right)^{q_R}
\right]
\left(\eta_{j_{p_R+1}}\right)
\right|
& > &
R^{p_R+q_R}
\,.
\end{eqnarray}
In particular, there are $p_1,\,q_1\geq0$ such that
\begin{eqnarray}\label{notcriter0}
\left|
\Delta_{p_1,(\eta_{j_{p_1}},\ldots,\eta_{j_1})}
\left[
\left(
\frac{\overline{\zeta}}{1+|\zeta|^2}
\right)^{q_1}
\right]
\left(\eta_{j_{p_1+1}}\right)
\right|
& \geq &
1
\,.
\end{eqnarray}
We set
\begin{eqnarray}\label{sigma1}
\sigma(k)
\;=\;
j_k
\;\mbox{ for all }\;
k=1,\ldots,p_1+1
& \mbox{ and } &
\sigma(p_1+2)
\;=\;
r_1
\,.
\end{eqnarray}
Then $\sigma$ is injective on the first $(p_1+2)$ indices,
$1$ is attained since $1\in\left\{j_1\,,\,r_1\right\}
\subset\sigma\left(\left\{1,\ldots,p_1+1,p_1+2\right\}\right)$
and~(\ref{notcriter0}) can be rewritten as
\begin{eqnarray}\label{notcriter1}
\left|
\Delta_{p_1,(\eta_{\sigma(p_1)},\ldots,\eta_{\sigma(1)})}
\left[
\left(
\frac{\overline{\zeta}}{1+|\zeta|^2}
\right)^{q_1}
\right]
\left(\eta_{\sigma(p_1+1)}\right)
\right|
& \geq &
1
\,.
\end{eqnarray}

The permutation $\sigma$ will be constructed by induction on $m\geq1$.
We first set
\begin{eqnarray}\label{convp0}
p_0
& := &
-2
\,,
\end{eqnarray}
and we assume having defined $\sigma$ on $\left\{1,\ldots,p_m+2\right\}$
where 
\begin{eqnarray}\label{pm-1pm}
p_{l-1}+2\;\leq\;p_l
& \mbox{ for all } &
l=1,\ldots,m
\,,
\end{eqnarray}
as follows: for all
$l=1,\ldots,m$,
\begin{eqnarray}\label{sigmam}
\sigma(k)
& = &
\begin{cases}
j_{k-l+1}
\;\mbox{ for all }\;
k=p_{l-1}+3,\ldots,p_l+1
\,,
\\
r_l
\;\mbox{ if }\;
k=p_l+2
\,.
\end{cases}
\end{eqnarray}
We also assume that for all $l=1,\ldots,m$,
\begin{eqnarray}\label{notcriterm}
\left|
\Delta_{p_l,(\eta_{\sigma(p_l)},\ldots,\eta_{\sigma(1)})}
\left[
\left(
\frac{\overline{\zeta}}{1+|\zeta|^2}
\right)^{q_l}
\right]
\left(\eta_{\sigma(p_l+1)}\right)
\right|
& \geq &
l^{p_l+q_l}
\,.
\end{eqnarray}
We indeed check that~(\ref{pm-1pm}) is fulfilled for $m=1$
since $p_1\geq0$ and $p_0=-2$ by~(\ref{convp0}).
Similarly, (\ref{sigmam}) (resp.~(\ref{notcriterm}))
is satisfied for $m=1$ by~(\ref{convp0}) and~(\ref{sigma1})
(resp. by~(\ref{notcriter1})).

Now let consider the sequence 
$\eta^{(m)}=\left(\eta_k^{(m)}\right)_{k\geq1}$ defined as follows:
\begin{eqnarray}\label{sigmam+10}
\eta_k^{(m)}
& := &
\begin{cases}
\eta_{\sigma(k)}
\;\mbox{ for all }\;
k=1,\ldots,p_m+2
\,,
\\
\eta_{j_{k-m}}
\;\mbox{ for all }\;
k\geq p_m+3
\,.
\end{cases}
\end{eqnarray}
Since $\left\{\eta^{(m)}\right\}$ (as a set) is the union of 
$\left\{\eta_{j_k}\right\}_{k\geq1}$ and the finite set
$\left\{\eta_{r_1},\ldots,\eta_{r_m}\right\}$ by~(\ref{sigmam}),
the sequence $\eta^{(m)}$ is bounded and satisfies~(\ref{remindnorealint}) as well
(with the same limit point $\zeta_0$).
Again, by successive applications of Proposition~\ref{specialequiv}
and~(\ref{thm12})$\Longleftrightarrow$(\ref{thm13}) from Theorem~\ref{theorem1}, 
it follows that $\eta^{(m)}$ satisfies~(\ref{notcriter}).
In particular, with the choice of $R=m+1$, there are
$p_{m+1},\,q_{m+1}\geq0$ such that
\begin{eqnarray}\label{notcriterm+1}
& &
\;\;\;\;
\left|
\Delta_{p_{m+1},\left(\eta^{(m)}_{p_{m+1}},\ldots,\eta^{(m)}_{1}\right)}
\left[
\left(
\frac{\overline{\zeta}}{1+|\zeta|^2}
\right)^{q_{m+1}}
\right]
\left(\eta^{(m)}_{p_{m+1}+1}\right)
\right|
\,\geq\,
(m+1)^{p_{m+1}+q_{m+1}}
.
\end{eqnarray}
In addition, we can choose 
$p_{m+1}\geq p_m+2$
(this will satisfy~(\ref{pm-1pm}) for all $l=1,\ldots,m+1$).
Indeed, if it were not possible, this would mean that
for all $R\geq m+1$, the associated $p_R$ should be bounded. 
By Corollary~\ref{corollairedulemme}, so would be all the terms
$\left|\Delta_{{p_R},\left(\eta^{(m)}_{p_R},\ldots,\eta^{(m)}_1\right)}
\left[\left(\dfrac{\overline{\zeta}}{1+|\zeta|^2}\right)^q\right]\left(\eta^{(m)}_{p_R+1}\right)\right|$ 
for all $q\geq0$, and this would contradict~(\ref{notcriter}) for $\eta^{(m)}$.
We can then extend $\sigma$ to $\left\{1,\ldots,p_{m+1}+2\right\}$ as follows:
\begin{eqnarray}\label{sigmam+1}
\sigma(k)
& = &
\begin{cases}
j_{k-m}
\;\mbox{ for all }\;
k=p_m+3,\ldots,p_{m+1}+1
\,,
\\
r_{m+1}
\;\mbox{ if }\;
k=p_{m+1}+2
\,.
\end{cases}
\end{eqnarray}
The induction hypothesis~(\ref{sigmam}) and~(\ref{sigmam+1}) show that $\sigma$ is well-defined on
$\left\{1,\ldots,p_{m+1}+2\right\}$. Moreover,
one has by~(\ref{sigmam+10}) and~(\ref{sigmam+1}) that
$\eta_{\sigma(k)}=\eta^{(m)}_{k}$ for all $k=1,\ldots,p_{m+1}+1$,
then it follows by~(\ref{notcriterm+1}) that~(\ref{notcriterm}) is still satisfied
for $l=m+1$. This last assertion with the induction hypotheses~(\ref{sigmam})
and~(\ref{notcriterm}) complete the case for $m+1$,
i.e. (\ref{sigmam}) and~(\ref{notcriterm}) are still satisfied for all
$l=1,\ldots,m+1$.

\medskip

The sequence $\left(p_m\right)_{m\geq1}$ constructed above allows us to
define $\sigma$ for all $k\geq1$ by~(\ref{sigmam}) since we have the following partition
from~(\ref{convp0}) and~(\ref{pm-1pm}),
\begin{eqnarray*}
\N\setminus\{0\}
& = &
\bigcup_{m\geq1}
\left\{k\,,\;
p_{m-1}+3\,\leq\,k\,\leq\,p_m+2
\right\}
\,.
\end{eqnarray*}
Next, $\sigma$ is a permutation of $\N\setminus\{0\}$: indeed, 
it follows from~(\ref{pm-1pm}) that every set
$\left\{k\,,\;
p_{m-1}+3\,\leq\,k\,\leq\,p_m+2
\right\}$ contains at least two elements, i.e. $\sigma$ attains by~(\ref{sigmam})
exactly one of the type
$r_m$ and at least one of the type $j_k$ as well.
On the other hand, one has by~(\ref{sigmam}) again that for all $m\geq1$,
$\sigma\left(p_m+1\right)=j_{p_m-m+2}$ and
$\sigma\left(p_m+3\right)=j_{p_m+3-(m+1)+1}=j_{p_m-m+3}$.
This last assertion and~(\ref{sigma1}) together show that all the $j_k$'s (resp.
$r_m$'s) are reached exactly once.

Finally, the estimate~(\ref{notcriterm}) being satisfied for all $m\geq1$
(i.e. this contradicts the estimate~(\ref{criter}) from
Theorem~\ref{theorem1}), it follows by
an application of Lemma~\ref{criternecessary} that there is
$f\in\mathcal{O}\left(\C^2\right)$ such that $R_N(f;\sigma(\eta))$
cannot be uniformly bounded (on any compact subset $K\subset\C^2$).
In particular, $R_N(f;\sigma(\eta))$ cannot even converge to $0$, then by Lemma~\ref{ENRN},
$E_N(f;\sigma(\eta))$ does not converge to $f$ (uniformly on any
compact subset $K\subset\C^2$). This achieves the
implication~(\ref{fusionpermut})$\implies$(\ref{fusiongeom})
from Theorem~\ref{fusion} and completes its whole proof.

\end{proof}

\medskip

\subsection{Proof of Corollaries~\ref{corollaireducorollaire} and~\ref{velocitat}}

In order to prove Corollary~\ref{corollaireducorollaire}, we first remind the following 
auxiliary result that is Lemma~8 from~\cite{amadeo4}.

\begin{lemma}\label{ballentire}

Let $r_0>0$ be fixed.
If there is $\varepsilon_{\eta}>0$ such that, $\forall\,f\in\mathcal{O}\left(B_2(0,r_0)\right)$,
$R_N(f;\eta)$ converges to $0$ uniformly on any compact subset of
$B_2(0,\varepsilon_{\eta}r_0)$, then
$\forall\,g\in\mathcal{O}\left(\C^2\right)$,
$R_N(g;\eta)$ converges to $0$ uniformly on any compact subset
of $\C^2$.

\end{lemma}

We can then give the proof of Corollary~\ref{corollaireducorollaire}.

\begin{proof}

First, as for the proof of Theorem~\ref{fusion},
the implication~(\ref{corolgeom})$\implies$(\ref{corolpermut})
(resp. (\ref{corolgeom})$\implies$(\ref{corolsubseq}))
from the statement of the corollary  immediately follows by
Theorem~\ref{theorem2} since, if
$\left\{\eta_j\right\}_{j\geq1}$ is locally interpolable by
real-analytic curve, then so will be the (same)
set $\left\{\sigma(\eta)\right\}=\left\{\eta_{\sigma(j)}\right\}_{j\geq1}$ for all
$\sigma\in\mathfrak{S}_{\N}$
(resp. the subset $\left\{\eta_{j_k}\right\}_{k\geq1}$
coming from any subsequence
$\left(\eta_{j_k}\right)_{k\geq1}$).

Conversely, in order to prove the implication~(\ref{corolpermut})$\implies$(\ref{corolgeom})
(resp.~(\ref{corolsubseq})$\implies$(\ref{corolgeom})),
let fix $\sigma\in\mathfrak{S}_{\N}$
(resp. $\eta'=\left(\eta_{j_k}\right)_{k\geq1}$)
and $g\in\mathcal{O}\left(\C^2\right)$.
The hypothesis~(\ref{corolpermut}) (resp.~(\ref{corolsubseq}))
and Lemma~\ref{ENRN} imply that
for all $f\in\mathcal{O}\left(B_2(0,r_0)\right)$,
$R_N(f;\sigma(\eta))$ (resp. $R_N\left(f;\eta'\right)$) converges
to $0$ uniformly on any compact subset
$K\subset B_2\left(0,\varepsilon_{\eta}r_0\right)$.
It follows by Lemma~\ref{ballentire} that in particular
$R_N(g;\sigma(\eta))$ (resp. $R_N\left(g;\eta'\right)$)
converges to $0$ uniformly on any compact subset
$K\subset\C^2$. Again, by an application of  Lemma~\ref{ENRN},
one can deduce that $E_N(g;\sigma(\eta))$
(resp. $E_N\left(g;\eta'\right)$) converges to $g$
uniformly on any compact subset $K\subset\C^2$.

Finally, $\sigma\in\mathfrak{S}_{\N}$
(resp. $\eta'=\left(\eta_{j_k}\right)_{k\geq1}$)
and $g\in\mathcal{O}\left(\C^2\right)$ being arbitrary,
the condition~(\ref{fusionpermut}) (resp.~(\ref{fusionsubseq})) from the statement
of Theorem~\ref{fusion} is satisfied, whose application yields
the required assertion~(\ref{corolgeom}).

\end{proof}

Lastly, the proof of Corollary~\ref{velocitat} immediately follows
by an application of Theorem~\ref{theorem2} (or
Theorem~3 from~\cite{amadeo4}) and Corollary~1 from~\cite{amadeo4}
since the set $\left\{\eta_{\sigma(j)}\right\}_{j\geq1}$
(resp. $\left\{\eta_{j_k}\right\}_{k\geq1}$) is still locally
interpolable by real-analytic curves.

\end{document}